\documentclass[a4paper]{amsart}

\usepackage{hyperref}
\usepackage{amsmath}
\usepackage{amssymb}
\usepackage{physics}
\usepackage{cmap}
\usepackage{dsfont}
\usepackage{bbm}
\usepackage{todonotes}
\usepackage[american]{babel}

\usepackage[]{parskip}

\theoremstyle{plain}
\newtheorem{theorem}{Theorem}[section]
\newtheorem{lemma}[theorem]{Lemma}
\newtheorem{proposition}[theorem]{Proposition}
\newtheorem{corollary}[theorem]{Corollary}
\theoremstyle{definition}
\newtheorem{definition}[theorem]{Definition}

\newtheorem{example}[theorem]{Example}
\newtheorem{assumption}[theorem]{Assumption}
\theoremstyle{remark}
\newtheorem{remark}[theorem]{Remark}

\DeclareMathOperator*{\maxgamma}{{\operatorname{max}_\gamma}}
\DeclareMathOperator*{\maxgammaone}{{\operatorname{max}_{\gamma_1}}}
\DeclareMathOperator*{\maxgammatwo}{{\operatorname{max}_{\gamma_2}}}
\DeclareMathOperator*{\sym}{sym}
\DeclareMathOperator*{\hto}{\overset{H}{\to}}

\DeclareMathOperator*{\identmatrix}{I}
\DeclareMathOperator*{\PVI}{P_{VI}}
\DeclareMathOperator*{\PCC}{P_{CC}}
\DeclareMathOperator*{\PGAMMA}{P_{\gamma}}

\title{Control in the Coefficients of an Obstacle Problem}
\author{Nicolai Simon}
\address{Universität Hamburg, MIN Faculty, Department of Mathematics, Bundesstr.~55, 20146 Hamburg, Germany}
\email{nicolai.simon@uni-hamburg.de}
\author{Winnifried Wollner}
\address{Universität Hamburg, MIN Faculty, Department of Mathematics, Bundesstr.~55, 20146 Hamburg, Germany}
\email[Corresponding author]{winnifried.wollner@uni-hamburg.de}

\date{\today}

\thanks{Funded by the Deutsche Forschungsgemeinschaft (DFG) --
  Projektnummer 314067056 within SPP 1962}

\begin{document}

\begin{abstract}
  In this work, we consider optimality conditions of an optimal control problem
  governed by an obstacle problem.
  Here, we focus on introducing a, matrix valued, control variable as the coefficients
  of the obstacle problem. As it is well known, obstacle problems can be formulated as
  a complementarity system and consequently the associated solution operator is not
  Gateaux differentiable. As a consequence, we utilize a regularization approach
  to obtain optimality conditions as the limit of optimality
  conditions of a family of
  regularized problems.

  Due to the coupling of the controlled coefficient with the gradients of the solution to
  the obstacle problem, weak convergence arguments can not be applied and we need to argue by
  $H$-convergence. We show, that, based on initial $H$-convergence, a bootstrapping argument
  can be utilized to prove strong $L^p$-convergence of the control and thus enable the passage
  to the limit in the optimality conditions.
\end{abstract}



\maketitle

\section{Introduction}
Obstacle problems are an example for variational inequalities and, as such, have been considered
by various researches, see, e.g,~\cite{KinderlehrerStampacchia:2000,Rodrigues:1987}
for an overview of some fundamental results.
From this research we know that the obstacle problem can equivalently be expressed as a
complementarity problem and further, as explained by Mignot in~\cite{Mignot:1976},
the solution operator is only directionally and not G\^ateaux differentiable. The presence of a
complementarity constraint prevents us from fulfilling standard constraint qualifications and
thus we cannot make use of standard optimality conditions. Instead, we consider alternative
stationarity concepts as discussed
in, e.g.,~\cite{ScheelScholtes:2000, SchielaWachsmuth:2011, Wachsmuth:2016}.
For the control of an obstacle by the right hand side it is possible to calculate generalized
derivatives, see,~\cite{RaulsUlbrich:2019,RaulsUlbrich:2021}. Their
results rely on monotonicity of
the active sets with respect to variations in the control.
Hence, in this work, we specifically focus on formulating Clarke-type first order limiting optimality
conditions, see~\cite{Clarke:1976}, by utilizing a regularized version of the problem, then
passing to the limit in the regularization parameter and considering the limits of the regularized
optimality conditions. Such a regularization approach is classical and has been employed for
obstacle problems without coefficient control in, e.g.,~\cite{Barbu:1984, Hintermueller:2001,MeyerRademacherWollner:2015, NeitzelWickWollner:2017,NeitzelWickWollner:2019, SchielaWachsmuth:2011}.

The novelty of this work is the introduction of a control in the coefficients of this problem,
which results in a coupling of control and state in the main part of the variational inequality.
The resulting problem statement is of particular interest in the study of inverse problems, where we
want to discover an optimal choice for the matrix parameters in a given problem without additional
requirements on the regularity of the control matrix, see~\cite{AsmannRoesch:2013, AsmannRoesch:2015, DeckelnickHinze:2011, DeckelnickHinze:2012}
for the use of such techniques in a PDE-constrained context or~\cite{P19Report2023} for an outlook on the use
of coefficient control in the context of variational inequalities.
It is known, that the limit of the product of two weakly converging sequences is not necessarily
the same as the product of their respective limits. Thus, instead of a standard weak limit
argument, we utilize $H$-convergence, see~\cite{MuratTartar:2018a, MuratTartar:2018b, Tartar:2018}, and a
variety of related results, see, e.g.,~\cite{Allaire:2002, HaslingerKocvaraLeugeringStingl:2010}
for an overview, when considering the limit. These concepts have previously been employed
for coefficient control of PDEs, see, e.g.,~\cite{DeckelnickHinze:2011, DeckelnickHinze:2012}, where
a useful projection formula for the feasible control set has been provided as well.
To the best of the authors knowledge, $H$-convergence has not been utilized for obtaining
limiting optimality conditions in the context of coefficient controls in variational inequalities.

The paper, provides the proofs for results announced in~\cite{SimonWollner:2023,P19Report2023} and is organized as follows.
In Section~\ref{sec:problem},
we establish the obstacle problem with control in the coefficients and introduce the notion
of $H$-convergence. In Section~\ref{sec:regularization}, we introduce the
regularized problem, present relevant properties and prove existence of solutions and
related necessary optimality conditions for the regularized problem.
Finally, in Section~\ref{sec:limitingcondition}, we will show the main result of the work
by utilizing a bootstrapping argument based on $H$-convergence to obtain first order limiting
optimality conditions for the obstacle problem with control in the coefficients.

\section{The Obstacle Problem with Control in the Coefficients}\label{sec:problem}
We consider an optimal control problem with a coefficient control governed by an obstacle
problem on a bounded domain $\Omega \subset \mathbb{R}^2$ that is Gröger-regular, see~\cite{Groger:1989}. A similar problem has been considered
in~\cite{P19Report2023}. It is given by
\begin{equation}\label{eq:obstacle_problem}\tag{$\PVI$}
  \begin{aligned}
    \min_{q, u} \; &J \left( q,u \right)
    = \frac{1}{2} \| u - u_d \|^2 + \frac{\alpha}{2} \left\| q \right\|^2\\
    \text{s.t. } &( q \nabla u, \nabla \left( v-u \right) ) \geq
    \left( f , v-u \right) \quad \forall v \in K \\
    &u \in K, \quad \quad \quad q \in Q^{\rm{ad}}
  \end{aligned}
\end{equation}
with $u_d \in L^2 \left( \Omega \right)$ and $f \in L^\infty(\Omega)$. The set of admissible states $K$ and
controls $Q^{\rm{ad}}$ is given by
\begin{equation*}
  \begin{aligned}
    K &= \left\{ v \in H_0^1 ( \Omega ) \,\big|\, v \geq \psi \text{
        q.e. on }\Omega \right\}\\
    \text{and } Q^{\rm{ad}} &= \left\{ q \in L^{2} \left( \Omega; \mathbb{R}_{\sym}^{2 \times 2} \right)
      \,\big|\, 0 \prec q_{\min} \identmatrix \preccurlyeq q(x) \preccurlyeq q_{\max}
      \identmatrix \text{ a.e. on }\Omega \right\}
  \end{aligned}
\end{equation*}
with given obstacle $\psi \in \mathbb{R}$ satisfying $\psi < 0$ and given bounds
$0 < q_{\min} < q_{\max} \in \mathbb{R}$ respectively.

All unspecified norms $\| \cdot \|$ and scalar products
$(\cdot,\cdot)$ are given in the respective $L^2$ spaces,
e.g. $L^2(\Omega)$, $L^2(\Omega;\mathbb{R}^2)$, $L^2(\Omega;\mathbb{R}^{2\times 2})$.
Further, $H^{-1}(\Omega)$ describes the dual space to $H^1_0 (\Omega)$ and
$\mathbb{R}^{2 \times 2}_{\text{sym}}$ is the set of of symmetric $2
\times 2$ matrices. Finally, $\curlyeqprec, \prec$ denote the
ordering of symmetric matrices, where $A \curlyeqprec B$ (or $A \prec
B$) holds iff $B-A$ is a positive semidefinite (or positive definite) matrix.

\begin{remark}
  The semidefinite cone constraints defined in $Q^{\rm{ad}}$ ensure that the
  control $q \in Q^{\rm{ad}}$ is a uniformly positive definite matrix function and the bounds
  further ensure that $q \in L^{\infty} \left( \Omega ; \mathbb{R}^{2 \times 2}_{\sym} \right)$.
  Crucially, as discussed in~\cite{P19Report2023}, one can utilize
  these bounds to show unique solvability of the obstacle problem
  \begin{equation}\label{eq:vi_obstacle_problem}
    \begin{aligned}
      ( q \nabla u, \nabla \left( v-u \right) ) \geq \left( f , v-u \right) \quad \forall v \in K.
    \end{aligned}
  \end{equation}
  for any $q \in Q^{\rm{ad}}$ and obtain the corresponding state $u
  \in K$. In addition, since
  $f \in L^\infty(\Omega)$ and $\psi \in \mathbb{R}$,
  we can apply~\cite[Chapter 5, Proposition 2.2]{Rodrigues:1987}
  to conclude that the solution satisfies $\nabla \cdot (q \nabla u)
  \in L^2(\Omega)$.
\end{remark}
Using the fact that $K - \psi$ is a convex cone, we can rewrite
problem~\eqref{eq:obstacle_problem}
using, e.g.,~\cite[Theorem~I.5.5]{KinderlehrerStampacchia:2000}, in
the equivalent complementarity formulation
\begin{equation}\label{eq:obstacle_problem_mpcc}\tag{$\PCC$}
  \begin{aligned}
    \min_{q \in Q^{\rm{ad}},u \in H_0^1 \left( \Omega \right) } \;&J \left( q,u \right) && \\
    \text{s.t. } &- \nabla \cdot \left( q \nabla u \right) = f + \lambda, &&\text{ in }H^{-1}(\Omega) \\
    &u \geq \psi && \text{ q.e. in } \Omega,\\
    &\lambda \geq 0 && \text{ in } H^{-1} \left( \Omega \right),\\
    &\left( \lambda, \psi - u \right) = 0.
  \end{aligned}
\end{equation}
Due to~\cite[Chapter~5, Proposition~2.2]{Rodrigues:1987} the
multiplier satisfies the additional regularity $\lambda \in L^2(\Omega)$.
Due to the multiplicative coupling of control and state, direct
methods to assert the existence of solutions
which rely on weak convergence of the variables in the
VI~\eqref{eq:vi_obstacle_problem} are no longer applicable.
Instead, we utilize $H$-convergence, as defined in~\cite[Section~2]{Tartar:2018}:
\begin{definition}[H-convergence]\label{def:h_convergence}
  A sequence $q_k \in Q^{\rm{ad}}$ $H$-converges to $q^H \in Q^{\rm{ad}}$
  (in symbols $q_k \hto q^H$) if
  \begin{equation*}
    \begin{aligned}
      q_k \nabla u_k \rightharpoonup q^H \nabla u  \text{ in } L^2 \left( \Omega, \mathbb{R}^2 \right)
    \end{aligned}
  \end{equation*}
  for any sequence $u_k \in H_0^1 \left( \Omega \right)$ satisfying
  \begin{equation*}
    \begin{aligned}
      u_k &\rightharpoonup  u &&\text{ in } H_0^1 \left( \Omega \right)\\
      \text{and } \nabla \cdot \left( q_k \nabla u_k \right) &\to g &&\text{ in } H^{-1} \left( \Omega \right)
    \end{aligned}
  \end{equation*}
  for some $u \in H_0^1 \left( \Omega \right)$ and $g \in H^{-1} \left( \Omega \right)$.
\end{definition}

\begin{remark}
  It should be noted, that we have chosen to define H-convergence
  based on its characteristic property stated
  in~\cite[Section~2]{Tartar:2018}, this is also often stated as a
  theorem of H-convergence for specific problems,
  e.g.,~\cite{DeckelnickHinze:2011, DeckelnickHinze:2012}. In the
  context given here, it is equivalent to the definition made in,
  e.g.,~\cite[Definition~1]{MuratTartar:2018b} or~\cite[Definition~1.2.15]{Allaire:2002}, when considering $g \in H^{-1}(\Omega)$ as the right hand side of a problem such that
  \begin{equation*}
    \begin{aligned}
      -\nabla \cdot \left( q \nabla u \right) = g
    \end{aligned}
  \end{equation*}
  which is the limit of the sequence $-\nabla \cdot \left( q_k \nabla u_k \right)$ for $k \to \infty$. This sequence can be rewritten with the corresponding right hand sides $g_k \in H^{-1}(\Omega)$, so that
  \begin{equation*}
    \begin{aligned}
      -\nabla \cdot \left( q_k \nabla u_k \right) = g_k.
    \end{aligned}
  \end{equation*}
  forms a sequence of problems as specified
  in~\cite[Section~2]{Tartar:2018}. For an explanation of how H-convergence induces this property, see,
  e.g.,~\cite[Theorem~1]{MuratTartar:2018b}. Equivalence then follows, by considering the case $g_k=g$.
\end{remark}

By using $H$-convergence the existence of solutions to problem~\eqref{eq:obstacle_problem_mpcc}
can be shown, see~\cite{P19Report2023}, allowing us to state the following existence result.
\begin{theorem}\label{thm:existence_obstacle_problem}
  There exists at least one solution of Problem~\eqref{eq:obstacle_problem_mpcc}.
\end{theorem}
\begin{proof}
  The proof to this theorem can be found in~\cite[Theorem~2.2]{P19Report2023}.
\end{proof}

\section{The Regularized Problem}\label{sec:regularization}
The control to state operator of an obstacle problem is, in general, not G\^ateaux differentiable,
see~\cite{Mignot:1976}. Therefore we will employ a regularization approach to compute
stationarity conditions for problem~\eqref{eq:obstacle_problem_mpcc} by considering the limit
points of stationarity conditions for a series of regularized problems similar to the work
of Meyer et al. in~\cite{MeyerRademacherWollner:2015}. In this section, we introduce the
regularized problems and some crucial bounds on the different variables.

We consider an optimal control problem governed by a regularized obstacle problem, given by
\begin{equation}\label{eq:reg_obstacle_problem}\tag{$\PGAMMA$}
  \begin{aligned}
    \min_{q_\gamma , u_\gamma} \quad
    &J\left( q_\gamma ,u_\gamma \right) = \frac{1}{2} \| u_\gamma - u_d \|^2
    + \frac{\alpha}{2} \left\| q_\gamma \right\|^2\\
    \text{s.t. } &- \nabla \cdot \left( q_\gamma \nabla u_\gamma \right)
    + r\left( \gamma ; u_\gamma \right) = f \quad \text{in }H^{-1}(\Omega) \\
    &u_\gamma \in H_0^1(\Omega), \quad q_\gamma \in Q^{\rm{ad}}
  \end{aligned}
\end{equation}
where $r: \mathbb{R}^{+} \times \mathbb{R} \to \mathbb{R}$ is
based on a $C^2$-approximation of a quadratic penalization of the
energy functional corresponding to~\eqref{eq:vi_obstacle_problem}, i.e.,
\begin{equation*}
  \begin{aligned}
    r\left( \gamma ; u_\gamma \right) := - \gamma \maxgamma \left( \psi - u_\gamma \right),
  \end{aligned}
\end{equation*}
with $\maxgamma(x)$ as a $C^2$-approximation of $\max \left( 0, x \right)$.

\begin{assumption}\label{rem:regularization}
  This approximation is inspired by the $C^2$-approximations
  from~\cite{KunischWachsmuth:2012a}, similar to them, we assume that
  the approximation of the maximum fulfills
  $\maxgamma : (\gamma, x) \mapsto \maxgamma (x)$ with
  $\gamma \in (0, \infty)$, $\maxgamma (x)  \geq 0$ and
  \begin{equation}\label{eq:smoothedMax_properties_leq}
    \begin{aligned}
      \maxgamma (x) \leq \max (0,x)
    \end{aligned}
  \end{equation}
  for all $x \in \mathbb{R}$. Further, we assume
  \begin{equation*}
    \begin{aligned}
      \maxgamma (x) = \max (0,x)
    \end{aligned}
  \end{equation*}
  for all $x$ with $x \geq \frac{1}{2\gamma}$ and all $x \leq
  0$. Finally, we assume that $\maxgamma\colon \mathbb{R}\rightarrow
  \mathbb{R}$ is continuously differentiable with respect to $x$,
  the derivative is bounded with
  \begin{equation}\label{eq:smoothedMax_properties_dxbound}
    \begin{aligned}
      0 \leq {\maxgamma}' (x) \leq 2
    \end{aligned}
  \end{equation}
  and the second derivative is bounded with
    \begin{equation*}
      \begin{aligned}
        0 \leq \left| {\maxgamma}^{''} (x) \right| \leq M \gamma
      \end{aligned}
    \end{equation*}
    for some constant $M > 0$.
    Finally we assume that there are constants $c_1,c_2 \geq 1$ such that
    \begin{equation}\label{eq:smoothedMax_deriv_relation}
      \begin{aligned}
        \frac{1}{c_1} {\maxgamma}' (x) x \leq \maxgamma (x) \leq c_2 {\maxgamma}' (x) x.
      \end{aligned}
    \end{equation}
\end{assumption}
\begin{example}
  For $x \in \mathbb{R}$ and $\gamma >0$, we denote with
  $m_\gamma (x): (0, \infty) \times \mathbb{R} \to \mathbb{R}$ a
  regularization of $x \mapsto \max (0,x)$ that fulfills the
  assumptions given in Remark~\ref{rem:regularization}. It is given by
  \begin{equation*}
    {\maxgamma} (x) :=
    \begin{cases} 48 \gamma^4 x^5 - 64 \gamma^3 x^4 + 24 \gamma^2 x^3
      &\text{ for } x \in \left( 0,\frac{1}{2\gamma} \right)\\
      \max (0,x) &\text{ else }
    \end{cases}
  \end{equation*}
  with derivatives
  \begin{equation*}
    {\maxgamma}^{'} (x) :=
    \begin{cases} 240 \gamma^4 x^4 - 256 \gamma^3 x^3 + 72 \gamma^2x^2
      &\text{ for } x \in \left( 0,\frac{1}{2\gamma} \right)\\
      \mathbbm{1}_{\{x > 0\}}(x) &\text{ else }
    \end{cases}
  \end{equation*}
  and
  \begin{equation*}
    {\maxgamma}^{''} (x) :=
    \begin{cases} 960 \gamma^4 x^3 - 768 \gamma^3 x^2 + 144 \gamma^2 x
      &\text{ for } x \in \left( 0,\frac{1}{2\gamma} \right)\\  0
      &\text{ else }
    \end{cases}.
  \end{equation*}
  Further, this example fulfills~\eqref{eq:smoothedMax_deriv_relation}
  with $c_1 = 7$ and $c_2 = 1$.
\end{example}
\begin{remark}
  Similar to the penalizations in,
  e.g.,~\cite{MeyerRademacherWollner:2015, SchielaWachsmuth:2011}, the resulting penalization
  \begin{equation*}
    \begin{aligned}
      H_0^1 \left( \Omega \right) \hookrightarrow  L^2 \left( \Omega \right)
      \ni u_\gamma \mapsto r \left( \gamma ; u_\gamma \right)
      \in L^2 \left( \Omega \right) \hookrightarrow  H^{-1} \left( \Omega \right)
    \end{aligned}
  \end{equation*}
  is locally Lipschitz continuous and monotone. Further, any feasible control
  is a uniformly positive definite
  and symmetric matrix function, therefore,
  given a control $q_\gamma \in Q^{\rm{ad}}$, the left-hand side
  of the PDE
  \begin{equation}\label{eq:pde_reg_obstacle_problem}
    \begin{aligned}
      - \nabla \cdot \left( q_\gamma \nabla u_\gamma \right) + r\left( \gamma ; u_\gamma \right) = f
    \end{aligned}
  \end{equation}
  is Lipschitz-continuous and monotone and coercive. Hence, the Browder-Minty theorem,
  see, e.g.,~\cite[Theorem~10.49]{RenardyRogers:2004},
  asserts, for each $q_\gamma \in Q^{\rm{ad}}$, the existence
  of a unique solution $u_\gamma \in H_0^1 ( \Omega )$.
\end{remark}

In effect, we can show that, for a given solution, the obstacle
still acts as a bound on the state,
which allows us to utilize the bounds on the control to compute
regularization independent bounds on
the gradient of the regularized state.
\begin{lemma}\label{lem:uniform_bound_l2_state}
  For every $q_\gamma \in Q^{\rm{ad}}$ the corresponding solution $u_\gamma \in H^1_0 (\Omega)$
  of~\eqref{eq:pde_reg_obstacle_problem} is bounded, with
  \begin{equation*}
    \| \nabla u_{\gamma} \| \lesssim \left\| f \right\|_{H^{-1}(\Omega)}
  \end{equation*}
  independent of $\gamma$ and $q_\gamma$. Here and throughout the
  manuscript $a\lesssim b$ is a shorthand notation for the
  inequality $a \le cb$ for a constant $c$
  independent of all relevant quantities.
\end{lemma}
\begin{proof}
  Using $u_\gamma$ as a test function we can utilize the lower bound $q_{\min}$ on the
  control $q_\gamma$ to estimate
  \begin{equation}
    \begin{aligned}\label{eq:p1-uniform_bound_l2_state}
      q_{\min} || \nabla u_{\gamma} ||^2  &\leq \left( q_{\gamma} \nabla u_{\gamma} , \nabla u_{\gamma} \right)\\
      &= \left(f , u_{\gamma} \right) - (r(\gamma; u_\gamma), u_\gamma).
    \end{aligned}
  \end{equation}
  Now we want to estimate $\left( r(\gamma; u_\gamma) , u_\gamma \right)$ by considering the
  regularization. Since this term is nonzero only if
  $\maxgamma \left( \gamma \left( \psi - u_\gamma \right) \right) >
  0$; and then
  $u_\gamma < \psi < 0$, we have $(r(\gamma; u_\gamma), u_\gamma) \geq 0$.
  This allows us to estimate
  \begin{equation}
    \begin{aligned}\label{eq:p2-uniform_bound_l2_state}
      \left(f , u_{\gamma} \right) - (r(\gamma; u_\gamma), u_\gamma) &\leq \left(f , u_{\gamma} \right)\\
      &\lesssim \left\| f \right\|_{H^{-1}(\Omega)} \left\| \nabla u_{\gamma} \right\|.
    \end{aligned}
  \end{equation}
  Finally combining~\eqref{eq:p1-uniform_bound_l2_state}
  and~\eqref{eq:p2-uniform_bound_l2_state} allows us to estimate
  \begin{equation*}
    q_{\min} \| \nabla u_\gamma \|^2 \lesssim \| f \|_{H^{-1}(\Omega)} \| \nabla u_\gamma \|.
  \end{equation*}
\end{proof}

Utilizing further properties of the obstacle, similar estimates can
also be achieved for the divergence and the regularization term
in~\eqref{eq:pde_reg_obstacle_problem}.
\begin{lemma}\label{lem:bounds_pde_reg_obstacle_problem}
  For every $q_\gamma \in Q^{\rm{ad}}$ and corresponding solution $u_\gamma \in H^1_0 (\Omega)$
  of~\eqref{eq:pde_reg_obstacle_problem}, we have
  $\nabla \cdot \left( q_\gamma \nabla u_\gamma \right) \in L^2 \left( \Omega, \mathbb{R}^2 \right)$.
  Further the bounds
  \begin{equation*}
    \begin{aligned}
      \| \nabla \cdot \left( q_\gamma \nabla u_\gamma \right) \| &\leq \| f \|\\
      \text{and } \quad \quad \| r ( \gamma , u_\gamma ) \| &\leq \| f \|
    \end{aligned}
  \end{equation*}
  are fulfilled independent of $\gamma$ and $q_\gamma$.
\end{lemma}
\begin{proof}
   First, we note that
  \begin{equation*}
    \begin{aligned}
      \nabla \cdot \left( q_\gamma \nabla u_\gamma \right) = r ( \gamma , u_\gamma ) - f \in L^2(\Omega)
    \end{aligned}
  \end{equation*}
  since $r(\gamma, u) \in L^2(\Omega)$ and $f \in L^\infty(\Omega)$. To compute an upper bound for our control, that is independent of the regularization parameter, we proceed with
  \begin{equation*}
    \begin{aligned}
      \| \nabla \cdot \left( q_\gamma \nabla u_\gamma \right) \|^2
      &= -\left( f - r \left( \gamma , u_\gamma \right), \nabla \cdot \left( q_\gamma \nabla u_\gamma \right) \right) \\
      &= -\left( f , \nabla \cdot q_\gamma \nabla u_\gamma \right)
      + \left( \gamma \maxgamma \left( \psi - u_\gamma \right), \nabla \cdot \left( q_\gamma \nabla u_\gamma \right) \right)\\
      &= -\left( f , \nabla \cdot q_\gamma \nabla u_\gamma \right)
      + \gamma \left( \nabla \maxgamma \left( \psi - u_\gamma  \right), q_\gamma \nabla u_\gamma \right),
    \end{aligned}
  \end{equation*}
  noting that $\psi \leq u_\gamma = 0$ on $\partial \Omega$ and thus
  $\maxgamma \left( \psi - u_\gamma \right) \in H^1_0(\Omega)$.
  Since $\psi \in \mathbb{R}$, we note that $\nabla \psi = 0$.
  Adding $0 =  \gamma \left( \nabla \maxgamma \left( \psi - u_\gamma \right), q_\gamma \nabla \psi \right)$,
  we get
  \begin{equation*}
    \begin{aligned}
      -\left( f , \nabla \cdot q_\gamma \nabla u_\gamma \right)
      &+ \gamma \left( \nabla \maxgamma  \left( \psi - u_\gamma  \right), q_\gamma \nabla u_\gamma \right)\\
      &= -\left( f , \nabla \cdot q_\gamma \nabla u_\gamma \right)
      + \gamma \left( \nabla \maxgamma  \left( \psi - u_\gamma \right), q_\gamma \nabla \left( u_\gamma - \psi \right) \right)\\
      &= -\left( f , \nabla \cdot q_\gamma \nabla u_\gamma \right)
      - \gamma \left( \nabla \maxgamma  \left( \psi - u_\gamma \right), q_\gamma \nabla \left( \psi - u_\gamma \right) \right)\\
      &= -\left( f , \nabla \cdot q_\gamma \nabla u_\gamma \right)
      - \gamma \left( {\maxgamma}'(\psi-u_\gamma)\nabla \left( \psi - u_\gamma \right), q_\gamma \nabla \left( \psi - u_\gamma \right) \right)\\
      &\leq -\left( f , \nabla \cdot q_\gamma \nabla u_\gamma \right)\\
      &\leq \| f \| \| \nabla \cdot \left( q_\gamma \nabla u_\gamma \right) \|.
    \end{aligned}
  \end{equation*}
  Here, the first inequality holds because ${\maxgamma}'  \left( \psi
    - u_\gamma \right) \ge 0$  and thus
  \begin{equation*}
    \left( {\maxgamma}'(\psi-u_\gamma)  \nabla \left( \psi - u_\gamma
      \right), q_\gamma \nabla \left( \psi-u_\gamma \right) \right)
    \geq 0.
  \end{equation*}
  Now we consider the regularization and proceed analogously. With this we get
  \begin{equation*}
    \begin{aligned}
      \| r ( \gamma , u_\gamma ) \|^2 &= -\left( r(\gamma; u_\gamma), \gamma \maxgamma  \left( \psi - u_\gamma \right)\right)\\
      &= -\left( f , \gamma \maxgamma  \left( \psi - u_\gamma \right) \right)
      - \gamma \left( \nabla \cdot q_\gamma \nabla u_\gamma , \maxgamma  \left( \psi - u_\gamma \right) \right)\\
      &= -\left( f , \gamma \maxgamma \left( \psi - u_\gamma \right) \right)
      + \gamma \left( q_\gamma \nabla u_\gamma , \nabla \maxgamma  \left( \psi - u_\gamma \right) \right)\\
      &= -\left( f , \gamma \maxgamma \left( \psi - u_\gamma \right) \right)
      - \gamma \left( q_\gamma \nabla (\psi -u_\gamma) , \nabla \maxgamma  \left( \psi - u_\gamma \right) \right)\\
      &\leq \| f \| \| r(\gamma ; u_\gamma) \|.
    \end{aligned}
  \end{equation*}
\end{proof}

By the assumed regularity in the sense of Gröger, see~\cite{Groger:1989},
we obtain improved integrability for the state $u_\gamma$.
\begin{lemma}\label{lem:uniform_bound_lp_state}
  There exists some $\widehat{p} > 2$ such that for any $p \in [2,\widehat{p})$
  and every $q_\gamma \in Q^{ad}$ the corresponding solution
  $u_\gamma \in H_0^1 ( \Omega )$ of~\eqref{eq:pde_reg_obstacle_problem}
  is bounded in $W^{1,p} (\Omega)$, with
  \begin{equation*}
    \| \nabla u_\gamma \|_{L^{p}( \Omega )} \lesssim \left\| f \right\|
  \end{equation*}
  independent of $\gamma$ and $q_\gamma$.
\end{lemma}
\begin{proof}
  First, we note that $u_\gamma$ solves the, linear, problem
  \begin{equation*}
    -\nabla \cdot (q_\gamma \nabla u_\gamma) = f - r(\gamma;u_\gamma)
  \end{equation*}
  with frozen nonlinearity. Hence, due to Gröger regularity,
  analogous to~\cite[Proposition 1.2]{HerzogMeyerWachsmuth:2011}, we can prove
  that the state solution $u_\gamma$ is in $W^{1,p}(\Omega)$ for all
  $2 \le p < \widehat{p}$ with a $\widehat{p}$ depending on the
  domain $\Omega$ and bounds $q_{\min}, q_{\max}$, only and satisfies the estimate
  \begin{equation*}
    \| \nabla u_\gamma \|_{L^p(\Omega)} \lesssim \|f-r(\gamma;u_\gamma)\|_{-1,p},
  \end{equation*}
  where $\|\cdot\|_{-1,p}$ denotes the norm on $(W_0^{1,p'})^*$, where
  $\frac{1}{p}+\frac{1}{p'} = 1$.
  To formulate an estimate for the $W^{-1,p}(\Omega)$ norm, we consider
  the Sobolev imbedding theorem, see,
  e.g.,~\cite[Theorem~4.12]{Adams:2003},
  which allows us to conclude that $W^{1,p'}_0 (\Omega) \subset
  W^{1,\widehat{p}'}_0(\Omega) \hookrightarrow
  L^2(\Omega)$ where the last imbedding is compact, since
  $1 - \frac{2}{\widehat{p}'} = 1 - (2 - \frac{2}{\widehat{p}}) =
  \frac{2}{\widehat{p}} -1 > -1 = 0 - \frac{2}{2}$.
  By duality, we assert
  \begin{equation*}
    L^2 (\Omega) = \left(L^2(\Omega) \right)^* \subset \left( W_0^{1,p'}
      ( \Omega ) \right)^*,
  \end{equation*}
  which allows us to conclude that $\| f-  r(\gamma;u_\gamma) \|_{-1,p} \lesssim \| f- r(\gamma;u_\gamma) \|$ and we can estimate
  \begin{equation*}
    \begin{aligned}
      \| \nabla u \|_{L^p(\Omega)} &\lesssim \|f-r(\gamma;u_\gamma)\|_{-1,p}\\
      &\lesssim \|f\| + \| r(\gamma;u_\gamma) \|\\
      &\lesssim \| f \|
    \end{aligned}
  \end{equation*}
  utilizing the bound on $\| r(\gamma;u_\gamma) \|$ from
  Lemma~\ref{lem:bounds_pde_reg_obstacle_problem}.
\end{proof}

Now we want to compute a bound on the feasibility violation for the solution to this problem following techniques from~\cite[Lemma 2.3]{SchielaWachsmuth:2011}.
\begin{lemma}\label{lem:growth_cond}
  For every $q_\gamma \in Q^{\textrm{ad}}$ with corresponding solution
  $u_\gamma \in H_0^1(\Omega)$ of~\eqref{eq:pde_reg_obstacle_problem},
  it holds
  \begin{equation*}
    \begin{aligned}
      \| \max ( \psi - u_\gamma, 0)\|_{L^s(\Omega)} &\lesssim \frac{1}{\gamma} (\|f\|_{L^s(\Omega)}+1)\\
      \| \nabla \max ( \psi - u_\gamma, 0)\| &\lesssim \frac{1}{\sqrt{\gamma}} (\|f\|+1)
    \end{aligned}
  \end{equation*}
  for all $s \in [2, \infty ]$.
\end{lemma}
\begin{proof}
  First, we note that by construction of $\maxgamma$ we have
  \begin{equation*}
    \begin{aligned}
      \max \left( \psi - u_\gamma , 0 \right) \leq \maxgamma \left( \psi - u_\gamma \right) + \frac{1}{2 \gamma}.
    \end{aligned}
  \end{equation*}
  Further, $u_\gamma$ is bounded in $W^{1,p}(\Omega)$, as proven in
  Lemma~\ref{lem:uniform_bound_lp_state}. Hence, we infer that $\max \left( \psi -
    u_\gamma , 0 \right) \in W^{1,p}_0(\Omega)$. Since $p>2$, we know
  that $W^{1,p}(\Omega)$ is a commutative Banach-Algebra~\cite[Theorem
  4.39]{Adams:2003} and thus we have $\max \left( \psi -
    u_\gamma , 0 \right)^{s-1} \in W^{1,p}_0$ as well.

  This allows us to estimate, for $s< \infty$,
  \begin{equation*}
    \begin{aligned}
      \gamma \| \max (\psi - u_\gamma , 0) \|^s_{L^s(\Omega)} &\leq \left( \gamma \maxgamma \left( \psi - u_\gamma \right) + \frac{1}{2} , \max \left( \psi - u_\gamma , 0 \right)^{s-1}  \right)\\
      &= \left( - \nabla \cdot \left( q_\gamma \nabla u_\gamma \right) - f + \frac{1}{2} , \max \left( \psi - u_\gamma , 0 \right)^{s-1}  \right)\\
      &\leq \left( \frac{1}{2} - f , \max \left( \psi - u_\gamma , 0 \right)^{s-1}  \right)\\
      &\leq \left\| \frac{1}{2} - f \right\|_{L^s(\Omega)} \left\| \max \left( \psi - u_\gamma , 0 \right)^{s-1}  \right\|_{L^{s'}(\Omega)}\\
     &= \left\| \frac{1}{2} - f \right\|_{L^s(\Omega)} \left\| \max \left( \psi - u_\gamma , 0 \right)  \right\|^{s-1}_{L^s(\Omega)}
    \end{aligned}
  \end{equation*}
  where $\frac{1}{s}+\frac{1}{s'}=1$.
  Note that the second inequality holds since $\nabla \psi = 0$, which allows us to estimate

  \begin{equation*}
    \begin{aligned}
      \big( - \nabla \cdot \left( q_\gamma \nabla u_\gamma \right), &\max \left( \psi - u_\gamma , 0 \right)^{s-1}  \big)\\
      &= \left( q_\gamma \nabla u_\gamma , \nabla \max \left( \psi - u_\gamma , 0 \right)^{s-1}  \right)\\
      &= \left( - q_\gamma \nabla \left( \psi - u_\gamma \right) , \nabla \max \left( \psi - u_\gamma , 0 \right)^{s-1}  \right)\\
      &= \left( - q_\gamma \nabla \max\left( \psi - u_\gamma,0 \right) , \nabla \max\left( \psi - u_\gamma,0 \right)^{s-1} \right)\\
      &\leq 0,
    \end{aligned}
  \end{equation*}
  where we used positive definiteness of $q$ in the last estimate.

  Finally dividing by $\gamma$ and $\left\| \max \left( \psi -
      u_\gamma , 0 \right)  \right\|^{s-1}_{L^s(\Omega)}$ gives the
  first inequality for $s < \infty$.
  Noting that the constant is independent of $s$, we can pass to the
  limit $s \to \infty$.

  For the second estimate, we consider
  \begin{equation*}
    \begin{aligned}
      q_{\min} \| \nabla \max (\psi - u_\gamma , 0) \|^2 &= q_{\min} \left( \nabla \max (\psi - u_\gamma, 0) , \nabla \max (\psi - u_\gamma , 0) \right)\\
      &= q_{\min} \left( \nabla (\psi - u_\gamma) , \nabla \max (\psi - u_\gamma , 0) \right)\\
      &\leq \left( q_\gamma \nabla (\psi - u_\gamma) , \nabla \max (\psi - u_\gamma , 0) \right)\\
      &= - \left( q_\gamma \nabla u_\gamma , \nabla \max (\psi - u_\gamma , 0) \right)\\
      &= - \left( f + \gamma \maxgamma (\psi - u_\gamma) , \max (\psi - u_\gamma , 0) \right)\\
      &\leq - \left( f + \gamma \max (\psi - u_\gamma, 0) - \frac{1}{2}, \max (\psi - u_\gamma , 0) \right)\\
      &= \left( \frac{1}{2} -f, \max (\psi - u_\gamma , 0) \right) - \gamma \left\| \max (\psi - u_\gamma, 0) \right\|^2\\
      &\leq \left\| \frac{1}{2} - f \right\| \left\| \max (\psi - u_\gamma , 0) \right\|.
    \end{aligned}
  \end{equation*}
  Thus the second estimate is shown.
\end{proof}

For this regularized problem we apply the notion of $H$-convergence, as used
in~\cite{P19Report2023}, to prove the existence of an optimal solution
to the optimal control problem.

\begin{theorem}\label{thm:existence_reg_obstacle_problem}
  There exists at least one solution for~\eqref{eq:reg_obstacle_problem}.
\end{theorem}
\begin{proof}
  We proceed similarly to the proof of the unregularized
  problem~\eqref{eq:obstacle_problem_mpcc} in~\cite[Theorem~2.2]{P19Report2023}.

  The objective function is bounded from below, so we can select a minimizing sequence
  $\left( q_k , u_k \right) \in Q^{\rm{ad}} \times H_0^1 \left( \Omega \right)$ satisfying the
  constraints in~\eqref{eq:reg_obstacle_problem}.
  The set $Q^{\rm{ad}}$ is sequentially compact with respect to $H$-convergence,
  see~\cite[Theorem~1.2.16]{Allaire:2002}, and the resulting convergence preserves the symmetry
  of the matrix, see \cite[Lemma~1.3.10]{Allaire:2002}, so we can select a subsequence, again
  denoted $q_k$, that $H$-converges to some limit $q^H$, i.e., $q_k \overset{H}{\to} q^H$.

  Since this sequence of $q_k$ is $H$-convergent, we show that the
  conditions for weak convergence
  of $q_k \nabla u_k$ are fulfilled. First, because the derivative $\nabla u_k$ is uniformly
  bounded, see Lemma~\ref{lem:uniform_bound_l2_state},
  we can conclude that we can find a subsequence again denoted $u_k$ that converges weakly,
  $u_k \rightharpoonup \bar{u}$, in $H^1_0 (\Omega)$. Secondly,
  the divergence, $ \nabla \cdot (q_k \nabla u_k)$, is bounded,
  see Lemma~\ref{lem:bounds_pde_reg_obstacle_problem}, so it converges weakly,
  up to selecting a subsequence, to some function $g$, i.e.,
  \begin{equation*}
    \begin{aligned}
      \nabla \cdot \left( q_k \nabla u_k \right) \rightharpoonup g \text{ in } L^2 \left( \Omega, \mathbb{R}^2  \right).
    \end{aligned}
  \end{equation*}
  The compactness of the imbedding then implies strong convergence in $H^{-1}(\Omega)$.
  Now, using $H$-convergence as in Definition~\ref{def:h_convergence}
  asserts the weak convergence
  \begin{equation*}
    \begin{aligned}
      q_k \nabla u_k \rightharpoonup q^H \nabla \bar{u} \text{ in } L^2 \left( \Omega, \mathbb{R}^2  \right).
    \end{aligned}
  \end{equation*}
  It remains to prove convergence of the regularization $r(\gamma; u_k)$.
  Since we only consider convergence for $k \to \infty$ with a
  given $\gamma$, convergence up to a subsequence follows immediately
  from strong $L^2(\Omega)$ convergence of $u_k$ and the, at most,
  linear growth of $r(\gamma; x)$.

  Finally, we note that the objective is lower-semicontinuous by definition in the first term
  and from~\cite{DeckelnickHinze:2011}, we know, that the second term is
  $H$-lower-semicontinuous. So we can infer the existence of a solution to the regularized
  optimal control problem~\eqref{eq:reg_obstacle_problem}.
\end{proof}
Finally we consider optimality conditions for this problem as have been considered in, e.g.,~\cite{DeckelnickHinze:2011,DeckelnickHinze:2012}.
\begin{proposition}\label{prop:opt_conditions_regularized}
  Let $\left( \bar{q}_{\gamma} , \bar{u}_{\gamma}\right) \in Q^{\rm{ad}} \times H^1_0(\Omega)$ be a
  local minimum of~\eqref{eq:reg_obstacle_problem}. Then there exists
  $\bar{p}_{\gamma} \in H_0^1(\Omega)$
  such that
  \begin{subequations}
    \label{eq:opt_conditions_regularized_All}
    \begin{align}
      - \nabla \cdot \left( \bar{q}_\gamma \nabla \bar{u}_{\gamma} \right)
      &= f - r ( \gamma, \bar{u}_{\gamma})& & \text{in }H^{-1}(\Omega), \label{eq:opt_conditions_regularized:1}\\
      - \nabla \cdot \left( \bar{q}_{\gamma} \nabla \bar{p}_{\gamma} \right)
      &= \bar{u}_{\gamma} - u_d- \partial_u r(\gamma; \bar{u}_{\gamma}) \bar{p}_{\gamma}& & \text{in }H^{-1}(\Omega), \label{eq:opt_conditions_regularized:2}\\
      \left( \alpha \bar{q}_\gamma - \nabla \bar{u}_\gamma \otimes \nabla \bar{p}_\gamma \right) (q - \bar{q}_\gamma)
      &\geq 0 &&\forall q \in Q^{\rm{ad}} \label{eq:opt_conditions_regularized:3}
    \end{align}
  \end{subequations}
  with $\nabla \bar{u}_\gamma \otimes \nabla \bar{p}_\gamma$
  denoting the outer product of $\nabla \bar{u}_\gamma$ and $\nabla \bar{p}_\gamma$.
\end{proposition}
\begin{proof}
  Let $(\bar{q}_\gamma, \bar{u}_\gamma)$ be an optimal solution to
  Problem~\eqref{eq:reg_obstacle_problem} and let
  \begin{equation*}
    S\colon L^{\infty}(\Omega;\mathbb{R}^{2\times 2}_{\rm{\sym}}) \supset
    Q^{\rm{ad}}\rightarrow W^{1,p}_0(\Omega)
  \end{equation*}
  be the control
  to state operator for the problem, particularly $\bar{u}_\gamma = S(\bar{q}_\gamma)$.
  To see differentiability of $S$, we first note that the Nemytskii
  operator is generated by a differentiable operator $\maxgamma
  \colon \mathbb{R} \rightarrow \mathbb{R}$. Hence, $\maxgamma \colon
  L^{\infty}(\Omega) \rightarrow L^\infty(\Omega)$ is G\^ateaux
  differentiable and differentiability as a mapping $W^{1,p}_0(\Omega)
  \rightarrow W^{-1,p}(\Omega)$ follows by compact embeddings. Now, the
  linearized PDE to~\eqref{eq:opt_conditions_regularized:1},
  w.r.t. $\bar{u}_\gamma$, defines an isomorphism $W^{1,p}_0(\Omega)
  \rightarrow W^{-1,p}(\Omega)$ following analogous arguments as in
  Lemma~\ref{lem:uniform_bound_lp_state} and noting that $0 \le \partial_u
  r(\gamma;\bar{u}_\gamma) \in L^\infty(\Omega)$.
  G\^ateaux differentiability of the solution
  operator then follows from the implicit function theorem.

  To derive the optimality conditions, we can express the objective in reduced form as
  \begin{equation*}
    \begin{aligned}
      j(\bar{q}_\gamma) = J(\bar{q}_\gamma, S(\bar{q}_\gamma)) = \frac{1}{2} \| S(\bar{q}_\gamma) - u_d \|^2 + \frac{\alpha}{2} \left\| \bar{q}_\gamma \right\|^2.
    \end{aligned}
  \end{equation*}
  For a local optimum, it is necessarily
  \begin{equation*}
    \begin{aligned}
      j'(\bar{q}_\gamma) (q - \bar{q}_\gamma) \geq 0 \quad\forall q \in Q^{ad}.
    \end{aligned}
  \end{equation*}
  Note that, our regularization approach ensures that the objective is G\^ateaux differentiable on $Q^{ad}$ with
  \begin{equation*}
    \begin{aligned}
      j'(\bar{q}_\gamma) (q - \bar{q}_\gamma) = \left(
        \bar{u}_\gamma - u_d, S'(\bar{q_\gamma})(q -
        \bar{q}_\gamma)\right) + \alpha \left( \bar{q}_\gamma, q -
        \bar{q}_\gamma \right) \quad \forall q \in Q^{ad}.
    \end{aligned}
  \end{equation*}
  By standard calculus, see,
  e.g.,~\cite{DeckelnickHinze:2011,DeckelnickHinze:2012,Troltzsch:2009},
  we obtain the adjoint equation
  \begin{equation}\label{eq:adjoint_equation_proof}
    \begin{aligned}
      \left( \bar{q}_\gamma \nabla \bar{p}_\gamma, \nabla v \right) + (\partial_u r(\gamma;  \bar{u}_\gamma) \bar{p}_\gamma , v) = (\bar{u}_\gamma- u_d , v) \quad\forall v \in H_0^1(\Omega)
    \end{aligned}
  \end{equation}
  and note that the linearization of the state equation in
  $\bar{u}_\gamma = S(\bar{q}_\gamma)$ in the direction
  $q-\bar{q}_\gamma$ satisfies, for any
  $v\in H^1_0(\Omega)$,
  \begin{equation}\label{eq:der_state_equation_proof}
    \begin{aligned}
      (\bar{q}_\gamma \nabla S'(\bar{q}_\gamma) (q - \bar{q}_\gamma), \nabla v)  + (\partial_u r(\gamma; \bar{u}_\gamma) &S'(\bar{q}_\gamma)(q - \bar{q}_\gamma),v)\\
      &= -((q- \bar{q}_\gamma) \nabla \bar{u}_\gamma, \nabla v)
    \end{aligned}
  \end{equation}
  Now, testing~\eqref{eq:adjoint_equation_proof} with
  $S'(\bar{q_\gamma})(q - \bar{q}_\gamma)$
  and~\eqref{eq:der_state_equation_proof} with $\bar{p}_\gamma$
  respectively, we get
  \begin{equation*}
    \begin{aligned}
      j'(\bar{q}_\gamma) (q - \bar{q}_\gamma) &= -(\nabla \bar{u}_\gamma, (q- \bar{q}_\gamma) \nabla \bar{p}_\gamma)+ \alpha (\bar{q}_\gamma , q - \bar{q}_\gamma)\\
      &= \left( - \nabla \bar{u}_\gamma \otimes \nabla \bar{p}_\gamma
        + \alpha \bar{q}_\gamma , q - \bar{q}_\gamma\right)  \quad \forall q \in Q^{ad}.
    \end{aligned}
  \end{equation*}
\end{proof}
To allow for a more direct comparison to an optimality system of
Problem~\eqref{eq:obstacle_problem_mpcc} when passing to the limit, we
introduce some additional variables
\begin{corollary}\label{crl:opt_conditions_regularized_extended}
  Let $\left( \bar{q}_{\gamma} , \bar{u}_{\gamma}\right) \in Q^{\rm{ad}} \times H^1_0(\Omega)$ be a
  local minimum of~\eqref{eq:reg_obstacle_problem}. Then there exist
  $\lambda_{\gamma}, \mu_{\gamma} \in L^2 (\Omega)$ and $p_{\gamma} \in H_0^1(\Omega)$
  such that
  \begin{subequations}
    \label{eq:extended_opt_conditions_regularized:All}
    \begin{align}
      - \nabla \cdot \left( \bar{q}_\gamma \nabla \bar{u}_{\gamma} \right)
      &= f - \bar{\lambda}_{\gamma}, \label{eq:extended_opt_conditions_regularized:1}\\
      \bar{\lambda}_{\gamma} - r ( \gamma, \bar{u}_{\gamma})
      &= 0,\label{eq:extended_opt_conditions_regularized:2}\\
      - \nabla \cdot \left( \bar{q}_{\gamma} \nabla \bar{p}_{\gamma} \right)
      &= \bar{u}_{\gamma} - u_d- \bar{\mu}_{\gamma},\label{eq:extended_opt_conditions_regularized:3}\\
      \bar{\mu}_{ \gamma } - \partial_u r( \gamma; \bar{u}_{\gamma}) \bar{p}_{ \gamma }
      &= 0,\label{eq:extended_opt_conditions_regularized:4} \\
      \left( \alpha \bar{q}_\gamma - \nabla \bar{u}_\gamma \otimes \nabla \bar{p}_\gamma \right) (q - \bar{q}_\gamma)
      &\geq 0 &&\forall q \in Q^{\rm{ad}}\label{eq:extended_opt_conditions_regularized:5}
    \end{align}
  \end{subequations}
  with $\nabla \bar{u}_\gamma \otimes \nabla \bar{p}_\gamma$ describing the outer product of $\nabla \bar{u}_\gamma$ and $\nabla \bar{p}_\gamma$.
\end{corollary}

\begin{remark}
  Note that the variational inequality
  condition~\eqref{eq:extended_opt_conditions_regularized:5} can be
  written as a projection onto the set $Q^{\rm{ad}}$ by utilizing techniques
  from~\cite[Section 2]{DeckelnickHinze:2011}.
  This allows us to conclude, that the stationarity points $\bar{q}_\gamma$ can
  be characterized equivalently as
  \begin{equation*}
    \bar{q}_\gamma (x) = P_{\rm{ad}} \left(  \frac{1}{\alpha} \nabla \bar{u}_\gamma (x) \otimes \nabla \bar{p}_\gamma (x) \right)
  \end{equation*}
  with $P_{\rm{ad}}: \mathbb{R}^{2 \times 2}_{\sym} \to \mathbb{R}^{2 \times 2}_{\sym}$ the orthogonal projection onto
  \begin{equation*}
    \mathcal{Q}_{ad} = \{ q \in \mathbb{R}^{2 \times 2}_{\sym} \,|\, q_{\min} \identmatrix \curlyeqprec q \curlyeqprec q_{\max} \identmatrix \}.
  \end{equation*}
  Also see \cite{P19Report2023} for the derivation of this formula given the problem considered in this paper.
\end{remark}

Similar to the state equation, we can also formulate bounds for the
adjoint equation.
\begin{lemma}\label{lem:uniform_bound_adjoint}
  Let $\bar{p}_\gamma \in H^1_0 (\Omega)$ be defined as specified in
  Proposition~\ref{prop:opt_conditions_regularized}. Then for every
  feasible point $(\bar{q}, \bar{u})$ the corresponding adjoint
  $\bar{p}_\gamma$ is bounded, with
  \begin{equation*}
    \begin{aligned}
      \| \nabla \bar{p}_\gamma \| &\lesssim \|f\|_{H^{-1} ( \Omega )}
      + \|u_d\|_{H^{-1}(\Omega)},
    \end{aligned}
  \end{equation*}
  independent of $\bar{q}_\gamma$ and $\gamma$.
\end{lemma}
\begin{proof}
  Testing the adjoint equation~\eqref{eq:opt_conditions_regularized:2} with
  $\bar{p}_{\gamma}$ yields
  \begin{equation*}
    \begin{aligned}
      &\left( \bar{q}_{\gamma} \nabla \bar{p}_{\gamma} , \nabla \bar{p}_{\gamma} \right)
      + \left( \partial_u r \left( \gamma ; \bar{u}_{\gamma} \right) \bar{p}_{\gamma}, \bar{p}_\gamma \right)
      = \left( \bar{u}_\gamma - u_d , \bar{p}_{\gamma} \right).
    \end{aligned}
  \end{equation*}
  Since $\partial_u r(\gamma; u_\gamma ) \in [0,2\gamma]$
  by~\eqref{eq:smoothedMax_properties_dxbound}, this allows us to find an
  upper bound for the adjoint using coercivity
  \begin{equation*}
    \begin{aligned}
      q_{\min} \left\| \nabla \bar{p}_\gamma \right\|^2 & \leq q_{\min} \left\| \nabla \bar{p}_\gamma \right\|^2 + \left( \partial_u r(\gamma; \bar{u}_\gamma) \bar{p}_\gamma , \bar{p}_\gamma \right)\\
      & \leq \left( \bar{q}_{\gamma} \nabla \bar{p}_\gamma, \nabla \bar{p}_\gamma \right) + \left( \partial_u r(\gamma; \bar{u}_\gamma) \bar{p}_\gamma , \bar{p}_\gamma \right)\\
      &= \left( \bar{u}_\gamma - u_d, \bar{p}_\gamma \right)\\
      &\leq \| \bar{u}_\gamma - u_d \|_{H^{-1}(\Omega)} \| \nabla \bar{p}_\gamma \|.
    \end{aligned}
  \end{equation*}
  Now, we can utilize the bound on $u_\gamma$ given by
  Lemma~\ref{lem:uniform_bound_l2_state} and
  $\|u_\gamma\|_{H^{-1}(\Omega)} \lesssim \|u_\gamma\|_{H^1(\Omega)}$ to conclude that the adjoint
  is uniformly bounded in $H^1(\Omega)$ independent of $\gamma > 0$.
\end{proof}

Finally, we can also formulate an $L^1(\Omega)$ estimate on the multiplier $\mu_\gamma$ using techniques from~\cite[Theorem 5.1]{ItoKunisch:2000}.
\begin{lemma}\label{lem:mu_bound}
  For any feasible point $(q_\gamma, u_\gamma)$, we can define an
  adjoint $p_\gamma \in H^1_0 (\Omega)$
  by~\eqref{eq:opt_conditions_regularized:2} in
  Proposition~\ref{prop:opt_conditions_regularized}.
  Then a multiplier $\mu_\gamma$ can be defined
  by~\eqref{eq:extended_opt_conditions_regularized:4} and satisfies
  \begin{equation*}
    \begin{aligned}
      \| \mu_\gamma \|_{L^1(\Omega)} \lesssim \left\| u_\gamma - u_d \right\|_{L^1 (\Omega)}.
    \end{aligned}
  \end{equation*}
\end{lemma}
\begin{proof}
  To see this estimate, we introduce the smoothed sign function
  \begin{equation*}
    \begin{aligned}
      S_{\varepsilon}(x) = \begin{cases} 1 &\text{ for } x \geq \varepsilon,\\ \frac{x}{\varepsilon} &\text{ for } | x | \leq \varepsilon,\\ -1 &\text{ for } x \leq - \varepsilon.
      \end{cases}
    \end{aligned}
  \end{equation*}
  With this we can use $S_{\varepsilon}(p_\gamma)$ as test function
  for~\eqref{eq:opt_conditions_regularized:2}, i.e., we get
  \begin{equation}\label{eq:l1_test}
    \begin{aligned}
      \left( q_\gamma \nabla p_\gamma, \nabla S_{\varepsilon}(p_\gamma) \right) + \left( \partial_u r( \gamma , u_\gamma) p_\gamma , S_{\varepsilon}(p_\gamma) \right) = \left( u_\gamma - u_d , S_{\varepsilon}(p_\gamma) \right).
    \end{aligned}
  \end{equation}
  Since  $S'_{\varepsilon}(p_\gamma) \geq 0$, we get
  \begin{equation*}
    \begin{aligned}
      \left( q_\gamma \nabla p_\gamma, \nabla S_{\varepsilon}(p_\gamma) \right) = \left( q_\gamma \nabla p_\gamma, S'_{\varepsilon}(p_\gamma) \nabla p_\gamma \right) \geq 0
    \end{aligned}
  \end{equation*}
  by positive definiteness of $q_\gamma$. So that~\eqref{eq:l1_test}
  and~\eqref{eq:extended_opt_conditions_regularized:4} give us
  \begin{equation*}
    \begin{aligned}
      \left( \mu_\gamma, S_{\varepsilon}(p_\gamma) \right) =
      \left( \partial_u r( \gamma , u_\gamma) p_\gamma , S_{\varepsilon}(p_\gamma) \right) \leq \left( u_\gamma - u_d , S_{\varepsilon}(p_\gamma) \right) \leq \| u_\gamma - u_d \|_{L^1 (\Omega)}
    \end{aligned}
  \end{equation*}
  where we utilize $\| S_{\varepsilon} (p_\gamma) \|_{L^\infty ( \Omega )} \leq 1$ in the last estimate.
  Now since $p_\gamma S_{\varepsilon}(p_\gamma) \to |p_\gamma| $
  almost everywhere on $\Omega$ as $\varepsilon \to 0$ and $0 \le \partial_ur(\gamma,u_\gamma) \leq 2\gamma$ we get
  \begin{equation*}
    \begin{aligned}
      \|\mu_\gamma\|_{L^1(\Omega)} = \left\| \partial_ur(\gamma,u_\gamma) p_\gamma \right\|_{L^1(\Omega)} \lesssim \left\| u_\gamma - u_d \right\|_{L^1 (\Omega)}
    \end{aligned}
  \end{equation*}
  by Lebesgue's dominated convergence theorem.
\end{proof}

\section{Limiting Optimality Conditions}\label{sec:limitingcondition}
We now compute first order limiting optimality conditions for the original problem.
First of all, we need to see that $H$-convergence is indeed compatible
with the regularization of the variational inequality.

\begin{lemma}\label{lem:feasibility_h_limit}
  Let $( \bar{q}_\gamma, \bar{u}_\gamma )$ be a sequence of feasible
  points of Problem~\eqref{eq:reg_obstacle_problem}. Assume that
  \[
    \bar{q}_\gamma \hto q^H
  \]
  as $\gamma \to \infty$. Then
  \begin{align*}
    \bar{u}_\gamma &\rightharpoonup
    \bar{u}_\infty& \text{in }&H^1_0(\Omega)\\
    -r(\gamma,\bar{u}_\gamma)& \rightharpoonup \bar{\lambda}_\infty& \text{in }&L^2(\Omega)
  \end{align*}
  and the pair
  $(q^H,\bar{u}_\infty,\bar{\lambda}_\infty)$ is feasible for
  Problem~\eqref{eq:obstacle_problem_mpcc}.
\end{lemma}
\begin{proof}
  Because $Q^{\rm{ad}}$ is $H$-sequentially compact by~\cite[Theorem~1.1.26]{Allaire:2002},
  the $H$-limit satisfies $q^H \in Q^{\rm{ad}}$.

  Lemma~\ref{lem:uniform_bound_l2_state} and
  Lemma~\ref{lem:bounds_pde_reg_obstacle_problem}
  assert the existence of a subsequence, denoted the same, and
  functions $\widetilde{u} \in H^1_0(\Omega)$, $\widetilde{\lambda} \in
  L^2(\Omega)$ such that
  \[
    \bar{u}_\gamma \rightharpoonup \widetilde{u}, \qquad
    -r(\gamma,\bar{u}_\gamma) \rightharpoonup \widetilde{\lambda}
  \]
  in $H^1_0(\Omega)$ and $L^2(\Omega)$, respectively.
  In addition, by Lemma~\ref{lem:uniform_bound_lp_state},
  $\bar{u}_\gamma \rightharpoonup \widetilde{u}$ holds in $W^{1,p}_0(\Omega)$
  and thus $\bar{u}_\gamma \rightarrow \widetilde{u}$ in $L^\infty(\Omega)$.
  Clearly, $\widetilde{\lambda} \ge 0$ as this is true for $-r(\gamma,\bar{u}_\gamma)$.

  By Lemma~\ref{lem:growth_cond}, we immediately see that
  \[
  \|\max(\psi-u_\gamma,0)\|_{L^\infty(\Omega)} \rightarrow 0
  \]
  and thus $\widetilde{u}\in K$.
  To see complementarity, we compute, utilizing strong convergence of $\bar{u}_\gamma$ in $L^\infty(\Omega)$,
  \begin{align*}
    0 &\ge (\widetilde{\lambda},\psi-\widetilde{u}) \\
    &\leftarrow -(r(\gamma;\bar{u}_\gamma),\psi-\bar{u}_\gamma) \\
    &= \gamma (\maxgamma(\psi-\bar{u}_\gamma),\psi-\bar{u}_\gamma)\\
    &\ge \gamma (\maxgamma(\psi-\bar{u}_\gamma),\maxgamma(\psi-\bar{u}_\gamma))\\
    &\ge 0
  \end{align*}
  utilizing~\eqref{eq:smoothedMax_properties_leq}.

  The Definition~\ref{def:h_convergence} of $H$-convergence asserts for any $\phi \in H^1_0(\Omega)$
  \[
    (q^H\nabla \widetilde{u},\nabla \phi) \leftarrow (\bar{q}_\gamma\nabla \bar{u}_\gamma,\nabla \phi) = (f,\phi) - (r(\gamma,\bar{u}_\gamma),\phi) \rightarrow (f,\phi) + (\widetilde{\lambda},\phi).
  \]
  Hence, we conclude that the triplet $(q^H,\widetilde{u},\widetilde{\lambda})$ is feasible
  for~\eqref{eq:obstacle_problem_mpcc}. Since, for fixed $q^H$ the
  pair $(\widetilde{u},\widetilde{\lambda})$ is uniquely determined the choice of a subsequence
  is not necessary proving the claim.
\end{proof}

We continue, showing that we need to assert that for any fixed control $q
\in Q^{\rm{ad}}$ the corresponding state solutions $u_\gamma$
of~\eqref{eq:pde_reg_obstacle_problem} converge to the respective
state solution $u$ of~\eqref{eq:vi_obstacle_problem} in $H^1_0(\Omega)$.
Since the control is fixed in this setting we can immediately
copy the arguments of~\cite[Lemma~3.3]{MeyerRademacherWollner:2015}
to see the assertion as the coefficient $q$ in the operator has no
influence on the arguments and the changes due to the different
regularization are minor.

\begin{lemma}\label{lem:convergence_state_setcontrol}
  Let $q \in Q^{\rm{ad}}$ be given and denote by
  $u_\gamma \in H_0^1(\Omega)$ the corresponding solution
  for~\eqref{eq:pde_reg_obstacle_problem} and by $u \in H_0^1(\Omega)$
  the corresponding solution of~\eqref{eq:vi_obstacle_problem}.
  Then
  \begin{equation*}
    u_\gamma \to u \text{ strongly in } H_0^1 ( \Omega
    ) \qquad \text{as } \gamma \to \infty.
  \end{equation*}
\end{lemma}
\begin{proof}
  We note, that the constant sequence $\bar{q}_\gamma = q $ is $H$-convergent to $q^H=q$. Hence
  the previous Lemma~\ref{lem:feasibility_h_limit} shows the weak convergence
  $u_\gamma \rightharpoonup u$.

  The mapping $u_\gamma \rightarrow \left( q \nabla u_\gamma , \nabla
    u_\gamma \right)$ is convex and continuous and thus weakly
    lower semicontinuous for a given fixed control $q$.
    We utilize this to show strong convergence
  \begin{equation*}
    \begin{aligned}
      0 &\leq \liminf_{\gamma \to \infty} \left( \left( q \nabla u_\gamma, \nabla u_\gamma \right) - \left( q \nabla u, \nabla u \right) \right)\\
      &\leq \limsup_{\gamma \to \infty} \left( \left( q \nabla u_\gamma, \nabla u_\gamma \right) - \left( q \nabla u, \nabla u \right) \right)\\
      &= \limsup_{\gamma \to \infty} \left( \left( q \nabla u_\gamma, \nabla u_\gamma \right) - \left( q \nabla u, \nabla u \right) - \left( q \nabla u_\gamma, \nabla u \right) + \left( q \nabla u, \nabla u_\gamma \right) \right)\\
      &= \limsup_{\gamma \to \infty} \left( q \nabla u, \nabla \left( u_\gamma - u \right) \right) - \left( q \nabla u_\gamma, \nabla \left( u - u_\gamma \right) \right)\\
      &\leq \lim_{\gamma \to \infty} \left( \left( q \nabla u, \nabla
          \left( u_\gamma - u \right) \right) - \left( f, u - u_\gamma
        \right) \right)\\
      &= 0.
    \end{aligned}
  \end{equation*}
  Note that we used the symmetry of $q$ to add $0 = \left(
    q \nabla u, \nabla u_\gamma \right) -
  \left( q \nabla u_\gamma , \nabla u \right)$ in the first equality
  and
  \[
  ( q \nabla u_\gamma, \nabla ( u - u_\gamma ) ) = (f,u - u_\gamma) - (r(\gamma;u_\gamma),u - u_\gamma) \ge (f,u - u_\gamma)
  \]
  in the last inequality which holds since $u \ge \psi \ge u_\gamma$ whenever $r(\gamma;u_\gamma)\ne 0$. This implies
  \begin{equation*}
    \lim_{\gamma \to \infty} \left( \left( q \nabla
        u_\gamma, \nabla u_\gamma \right) -
      \left( q \nabla u, \nabla u \right)
    \right) = 0
  \end{equation*}
  and since the scalar product $(q\nabla \cdot,\nabla \cdot) $ is
  equivalent to the $H^1_0(\Omega)$ scalar product this shows
  $\left\| u_\gamma \right\|_{H^1(\Omega)} \to \left\| u \right\|_{H^1(\Omega)}$.

  Together with the weak convergence of $u$ this proves strong convergence in $H_0^1(\Omega)$.
\end{proof}

As a basis for further arguments, we first show that a series of optimal solutions of the
regularized problems has an $H$-accumulation point that is a global minimizer of the original problem.

\begin{theorem}\label{thm:h_limit_optimal}
  Let $( \bar{q}_\gamma, \bar{u}_\gamma )$ be a global minimizer of
  Problem~\eqref{eq:reg_obstacle_problem} and $(\bar{q}, \bar{u})$ a
  global minimizer of Problem~\eqref{eq:obstacle_problem_mpcc}. If
  $\gamma \to \infty$ then every sequence $\bar{q}_{\gamma}$ admits a
  $H$-accumulation point, and any such accumulation point is a global
  minimizer of~\eqref{eq:obstacle_problem_mpcc}. Further, every sequence of regularized
  solutions has a subsequence $(\bar{q}_\gamma, \bar{u}_\gamma)$ such that
  \begin{equation*}
    J(\bar{q}_\gamma, \bar{u}_\gamma) \to J(\bar{q}, \bar{u}).
  \end{equation*}
\end{theorem}

\begin{proof}
  Because $Q^{\rm{ad}}$ is $H$-sequentially compact by~\cite[Theorem~1.2.16]{Allaire:2002},
  there is a subsequence of $\bar{q}_{\gamma}$ again denoted $\bar{q}_{\gamma}$ that $H$-converges to
  some $q^H$ in $Q^{\rm{ad}}$. By
  Lemma~\ref{lem:uniform_bound_l2_state} the
  state $\bar{u}_{\gamma}$ is bounded in $H^1_0 ( \Omega)$ and thus has a weakly convergent
  subsequence again denoted $\bar{u}_{\gamma}$ with weak limit
  $\bar{u}_\infty$ in $H^1_0(\Omega)$. The tracking-type functionals considered here are lower
  semicontinuous and, see~\cite{DeckelnickHinze:2011}, the second term is also
  $H$-lower semicontinuous, so this gives us
  \begin{equation*}
    J \left( q^H , \bar{u}_\infty \right) \leq \liminf_{\gamma \to \infty} J \left( \bar{q}_{\gamma}, \bar{u}_{\gamma} \right).
  \end{equation*}

  Now let $\bar{q}$ be a, globally, optimal control of Problem~\eqref{eq:obstacle_problem_mpcc} and
  let $\bar{u}$ be the corresponding state variable, then, noting that
  $(q^H,\bar{u}_\infty)$ is feasible
  for~\eqref{eq:obstacle_problem_mpcc} by
  Lemma~\ref{lem:feasibility_h_limit}, we have
  \begin{equation*}
    \begin{aligned}
      J \left( \bar{q} , \bar{u} \right) &\leq J \left( q^H , \bar{u}_\infty \right)\\
      &\leq \liminf_{\gamma \to \infty} J \left( \bar{q}_{\gamma}, \bar{u}_{\gamma} \right)\\
      &\leq \limsup_{\gamma \to \infty} J \left( \bar{q}_{\gamma}, \bar{u}_{\gamma} \right).
    \end{aligned}
  \end{equation*}
  Now let $\widehat{u}_{\gamma}$ be the corresponding state to control
  $\bar{q}$ for problem~\eqref{eq:reg_obstacle_problem}. Then, because
  $\left( \bar{q}_\gamma , \bar{u}_\gamma \right)$ is the optimal solution
  to~\eqref{eq:reg_obstacle_problem} and $\left( \bar{q} , \widehat{u}_\gamma \right)$ is a
  feasible point for this problem, we get
  \begin{equation*}
    \begin{aligned}
      \limsup_{\gamma \to \infty} J \left( \bar{q}_{\gamma}, \bar{u}_{\gamma} \right)
      \leq \lim_{\gamma \to \infty} J \left( \bar{q}, \widehat{u}_{\gamma} \right)
    \end{aligned}
  \end{equation*}
  For a given control, the corresponding state converges strongly in
  $H^1_0 (\Omega)$, see Lemma~\ref{lem:convergence_state_setcontrol}.
  With this we can estimate
  \begin{equation*}
    \begin{aligned}
      \lim_{\gamma \to \infty} J \left( \bar{q}, \widehat{u}_{\gamma} \right) = J \left( \bar{q}, \bar{u} \right).
    \end{aligned}
  \end{equation*}
  Thus, we conclude that
  \begin{equation*}
    J \left( q^H, \bar{u}_\infty \right) = \lim_{\gamma \to \infty} J(\bar{q}_{\gamma} , \bar{u}_{\gamma}) = J \left( \bar{q}, \bar{u} \right)
  \end{equation*}
  and hence $(q^H, \bar{u}_\infty)$ is indeed a global minimizer of
  problem~\eqref{eq:obstacle_problem_mpcc}.
\end{proof}
Note that without further arguments $H$-convergence only provides the weak convergence
\begin{equation*}
  \begin{aligned}
    \bar{q}_{\gamma} \nabla \bar{u}_{\gamma} \rightharpoonup q^H \nabla \bar{u}  \text{ in } L^2 \left( \Omega , \mathbb{R}^2 \right)
  \end{aligned}
\end{equation*}
where the control $\bar{q}_\gamma$ $H$-converges towards the $H$-limit $q^H$.
When considering optimality conditions, it is important to note that
such an $H$-limit of the control is not always equivalent to the
corresponding weak limit,
however, using the properties of the problem considered in this work,
we can bootstrap $H$-convergence, utilizing techniques from \cite[Theorem 3.2]{DeckelnickHinze:2011}, to prove
stronger regularity of the control and the state.

\begin{theorem}\label{thm:convergence_regularization}
  If $\gamma \to \infty$, then there is a subsequence of solutions
  $(\bar{q}_\gamma , \bar{u}_\gamma) \in Q^{\rm{ad}} \times H^1_0 (\Omega)$ to
  problem~\eqref{eq:reg_obstacle_problem}, with corresponding adjoint
  $\bar{p}_{\gamma} \in H_0^1(\Omega)$ as defined in
  Proposition~\ref{prop:opt_conditions_regularized}, such that
  \begin{subequations}
    \label{eq:convergence_regularization:All}
    \begin{align}
      \bar{q}_{\gamma} &\to \bar{q}  &&\text{ in } L^p \left( \Omega, \mathbb{R}^{2 \times 2}_{\text{sym}} \right) \text{ for all } 2 \leq p < \infty, \label{eq:convergence_regularization:1}\\
      \bar{u}_\gamma &\to \bar{u} &&\text{ in } W^{1,p} \left( \Omega
      \right) \text{ for all } 2 \le p < \widehat{p}, \label{eq:convergence_regularization:2}\\
      \bar{p}_{\gamma} &\rightharpoonup \bar{p} &&\text{ in }  W^{1,2} \left( \Omega \right), \label{eq:convergence_regularization:3}
    \end{align}
  \end{subequations}
  for some $(\bar{q}, \bar{u}, \bar{p}) \in Q^{\rm{ad}} \times H^1_0 (\Omega) \times H^1_0 (\Omega)$.
\end{theorem}
\begin{proof}
  First we prove strong $L^2(\Omega, \mathbb{R}^{2 \times
    2}_{\text{sym}})$ convergence for a subsequence of the
  control. Let $q^H$ be the H-limit and let $\bar{u}$ be the
  $L^2(\Omega)$ limit of subsequences of control $\bar{q}_\gamma$ and
  state $\bar{u}_\gamma$ respectively.
  To consider norm convergence, we further note that the control is
  weak$*$ convergent to some $q^*$ in $L^\infty
  (\Omega, \mathbb{R}^{2\times 2}_{\sym})$ which gives us $(q_H \xi, \xi) \leq (q^*
  \xi, \xi)$ for all $\xi \in \mathbb{R}^2$,
  by~\cite[Theorem~5]{Tartar:2018} or~\cite[Theorem~1.2.14]{Allaire:2002}.
  In particular, this allows us to estimate that $\|q^H\|^2 \leq (q^H,q^*)$ since
  \begin{equation*}
    \begin{aligned}
      0 \leq \trace(q^H(q^*-q^H)) = \trace (q^H q^*) - \trace(q^H q^H) = (q^H, q^*) - \|q^H\|^2,
    \end{aligned}
  \end{equation*}
  where the first inequality holds by~\cite[Theorem
  7.5.4]{HornJohnson:1990}.

  Based on this estimate we can prove norm convergence by the
  following calculations
  \begin{equation*}
    \begin{aligned}
      \frac{1}{2} \left\| \bar{u} - \bar{u}_{\gamma} \right\|^2 &+ \frac{\alpha}{2} \left\| q^H - \bar{q}_{\gamma} \right\|^2\\
      &= \frac{1}{2} \left\| (\bar{u} - u_d) - (\bar{u}_{\gamma} - u_d) \right\|^2 + \frac{\alpha}{2} \left\| q^H - \bar{q}_{\gamma} \right\|^2\\
      &= \frac{1}{2} \left\| \bar{u} - u_d \right\|^2 - \left( (\bar{u} - u_d), (\bar{u}_{\gamma} - u_d) \right) + \frac{1}{2} \left\| \bar{u}_{\gamma} - u_d \right\|^2\\
      &\quad + \frac{\alpha}{2} \left\| q^H \right\|^2 - \alpha \left(q^H, \bar{q}_\gamma \right) + \frac{\alpha}{2} \left\|\bar{q}_{\gamma} \right\|^2\\
      &= J\left( q^H, \bar{u} \right) + J\left( \bar{q}_\gamma , \bar{u}_\gamma  \right)- \left( (\bar{u} - u_d), (\bar{u}_{\gamma} - u_d) \right) - \alpha \left(q^H, \bar{q}_\gamma \right)\\
      &\to 2 J\left( q^H, \bar{u} \right) - \left\| \bar{u} - u_d \right\|^2 - \alpha \left(q^H, q^* \right)\\
      &\leq 2 J\left( q^H, \bar{u} \right) - \left\| \bar{u} - u_d \right\|^2 - \alpha \left\|q^H \right\|^2\\
      &= 2 J\left( q^H, \bar{u} \right) - 2 J\left( q^H, \bar{u} \right)\\
      &= 0.
    \end{aligned}
  \end{equation*}
  For the limit in this argument, we consider that by
  Theorem~\ref{thm:h_limit_optimal} the $H$-limit is a global optimum
  of the original problem, and
  \begin{equation*}
    \begin{aligned}
      J(\bar{q}_\gamma, \bar{u}_\gamma) \to J(q^H, \bar{u}).
    \end{aligned}
  \end{equation*}
  Further, since the state $\bar{u}_\gamma \rightharpoonup \bar{u}$
  weakly in $H^1_0(\Omega)$ by Lemma~\ref{lem:uniform_bound_l2_state},
  the compact imbedding $H^1_0(\Omega) \hookrightarrow L^2(\Omega)$ gives
  \begin{equation*}
    \left\| \bar{u}_\gamma - u_d \right\| \to \left\| \bar{u} - u_d \right\| \text{ in }L^2(\Omega).
  \end{equation*}

  With the norm convergence we can now conclude that the subsequence
  $\bar{q}_\gamma$ converges strongly in $L^2(\Omega, \mathbb{R}^{2
    \times 2}_{\text{sym}})$, resulting in
  \begin{equation*}
    \bar{q}_\gamma \to q^H = q^* \text{ in } L^2 (\Omega, \mathbb{R}^{2 \times 2}_{\text{sym}}),
  \end{equation*}
  since strong convergence ensures that both limits coincide,
  see~\cite[Lemma~1.2.22]{Allaire:2002}.

  Given the strong $L^2(\Omega, \mathbb{R}^{2 \times 2}_{\text{sym}})$
  convergence of the control variable, we can now utilize the bounds on the set
  of admissible controls, i.e., $Q^{\rm{ad}}$, which ensure that
  $\bar{q}_\gamma \in L^\infty (\Omega, \mathbb{R}^{2 \times 2}_{\text{sym}})$
  independent of $\gamma$. Thus, utilizing the Riesz-Thorin theorem, we get
  \begin{equation*}
    \bar{q}_\gamma \to \bar{q} \text{ in } L^p (\Omega, \mathbb{R}^{2 \times 2}_{\text{sym}}) \text{ for all } 2 \leq p < \infty.
  \end{equation*}

  Now considering the state $\bar{u}_\gamma$, we utilize the
  $W^{1,t}(\Omega)$ bound, for any $t < \widehat{p}$, from Lemma~\ref{lem:uniform_bound_lp_state}
  to show that $\bar{u}_\gamma$ is a Cauchy sequence in
  $W^{1,p}(\Omega)$ for all $p < t$. First we compute:
  \begin{equation*}
    \begin{aligned}
      (\bar{q}_{\gamma_1} &\nabla \left( \bar{u}_{\gamma_1} - \bar{u}_{\gamma_2} \right), \nabla \varphi )\\
      &= \left( \bar{q}_{\gamma_1} \nabla \bar{u}_{\gamma_1} - \bar{q}_{\gamma_2} \nabla \bar{u}_{\gamma_2}
        + \bar{q}_{\gamma_2} \nabla \bar{u}_{\gamma_2} - \bar{q}_{\gamma_1} \nabla \bar{u}_{\gamma_2}, \nabla \varphi  \right)\\
      &= \left( \bar{q}_{\gamma_1} \nabla \bar{u}_{\gamma_1} , \nabla \varphi \right) - \left( \bar{q}_{\gamma_2} \nabla \bar{u}_{\gamma_2} , \nabla \varphi \right)
      + \left( \bar{q}_{\gamma_2} \nabla \bar{u}_{\gamma_2} , \nabla \varphi \right) - \left( \bar{q}_{\gamma_1} \nabla \bar{u}_{\gamma_2} , \nabla \varphi \right)\\
      &= \left( f , \varphi \right) + \left( \gamma_1 \maxgammaone \left( \psi - \bar{u}_{\gamma_1} \right) , \varphi \right)
      - \left( f, \varphi \right) - \left( \gamma_2 \maxgammatwo \left( \psi - \bar{u}_{\gamma_2} \right) , \varphi \right)\\
      &\quad \quad + \left( \left( \bar{q}_{\gamma_2} - \bar{q}_{\gamma_1} \right) \nabla \bar{u}_{\gamma_2} , \nabla \varphi \right)\\
      &= \left( \gamma_1 \maxgammaone \left( \psi - \bar{u}_{\gamma_1} \right)
        - \gamma_2 \maxgammatwo \left( \psi - \bar{u}_{\gamma_2} \right) , \varphi \right)
      + \left( \left( \bar{q}_{\gamma_2} - \bar{q}_{\gamma_1} \right) \nabla \bar{u}_{\gamma_2} , \nabla \varphi \right)
    \end{aligned}
  \end{equation*}
  Because $\bar{q}_{\gamma_1}$ is bounded, we proceed analogously to the proof of Lemma~\ref{lem:uniform_bound_lp_state} and utilize Gröger-regularity to formulate the following upper bound
  \begin{equation*}
    \begin{aligned}
      \| \nabla \left( \bar{u}_{\gamma_1} - \bar{u}_{\gamma_2} \right) &\|_{L^p ( \Omega )}\\
      &\lesssim \| r(\gamma_1; \bar{u}_{\gamma_1}) - r(\gamma_2; \bar{u}_{\gamma_2}) \|_{W^{-1,p}(\Omega)}
      + \| \bar{q}_{\gamma_2} - \bar{q}_{\gamma_1} \|_{L^r(\Omega)} \| \nabla \bar{u}_{\gamma_2} \|_{L^t(\Omega)}
    \end{aligned}
  \end{equation*}
  with $\frac{1}{r} + \frac{1}{t} = \frac{1}{p}$.

  By Lemma~\ref{lem:uniform_bound_lp_state} and
  $\bar{q}_\gamma$ converges in $L^r (\Omega)$ for all $2 \leq r <
  \infty$, hence the above $t$ can be chosen to satisfy $t \in (p,\widehat{p})$ and thus
  $\bar{u}_\gamma$ is bounded in $W^{1,t} (\Omega)$.
  Hence, by convergence of $q_\gamma$ the last summand converges to zero.

  Considering the regularization term, we utilize the Sobolev
  imbedding theorem to conclude that the imbedding $W^{1, p}
  (\Omega)\subset L^2 (\Omega)$ is compact. By
  Lemma~\ref{lem:bounds_pde_reg_obstacle_problem} we get an $L^2
  (\Omega)$ bound for the regularization and thus weak convergence in
  $L^2(\Omega)$ of an appropriate subsequence. By compact imbedding we
  then get strong convergence of the regularization terms and can
  conclude that
  \begin{equation*} \| \gamma_1 \maxgammaone \left( \psi - \bar{u}_{\gamma_1} \right)
    - \gamma_2 \maxgammatwo \left( \psi - \bar{u}_{\gamma_2} \right)
    \|_{W^{-1,p}} \to 0.
  \end{equation*}
  Therefore we can conclude that
  \begin{equation*}
    \| \nabla \left( \bar{u}_{\gamma_1} - \bar{u}_{\gamma_2} \right)
    \|_{L^p ( \Omega )} \to 0,
  \end{equation*}
  so $\bar{u}_\gamma$ is
  a Cauchy-sequence and converges in $W^{1,p}(\Omega)$.

  Weak convergence of the adjoint follows directly from its uniform boundedness as specified in Lemma~\ref{lem:uniform_bound_adjoint}.

\end{proof}

These improved regularity results then allow us to prove strong
convergence results for the coupled terms of control and state, as well as the Lagrange multiplier
$\lambda_\gamma$ as defined in
Corollary~\ref{crl:opt_conditions_regularized_extended}.

\begin{corollary}\label{crl:extended_convergence_regularization}
  If $\gamma \to \infty$, then there is a subsequence of solutions
  $(\bar{q}_\gamma , \bar{u}_\gamma) \in Q^{\rm{ad}} \times H^1_0 (\Omega)$ to
  problem~\eqref{eq:reg_obstacle_problem} as defined in
  Proposition~\ref{prop:opt_conditions_regularized} and multipliers
  $\lambda_\gamma \in L^2 (\Omega)$ as defined in
  Corollary~\ref{crl:opt_conditions_regularized_extended}, such that
  \begin{subequations}
    \label{eq:extended_convergence_regularization:All}
    \begin{align}
      \bar{q}_{\gamma} \nabla \bar{u}_\gamma &\to \bar{q} \nabla \bar{u} &&\text{ in } L^2 \left( \Omega, \mathbb{R}^2 \right), \label{eq:extended_convergence_regularization:4}\\
      \bar{\lambda}_{\gamma} &\to \bar{\lambda} &&\text{ in } L^{2} \left( \Omega \right), \label{eq:extended_convergence_regularization:7}
    \end{align}
  \end{subequations}
  for some $(\bar{q}, \bar{u}, \bar{\lambda}) \in
  Q^{\rm{ad}} \times H^1_0 (\Omega) \times
  H^{-1}(\Omega)$.
\end{corollary}
\begin{proof}
  The results of this corollary are a based on the results from
  Theorem~\ref{thm:convergence_regularization}.

  From~\eqref{eq:convergence_regularization:2} we know that the state
  $\bar{u}_\gamma$ converges strongly in $W^{1,p}$ for some $p>2$. Further
  from~\eqref{eq:convergence_regularization:1} we know that $\bar{q}_\gamma$
  converges strongly in $L^s(\Omega, \mathbb{R}^{2 \times
    2}_{\text{sym}})$ for all $s \in [2, \infty)$. In particular, the
  control converges strongly for an $s$ such that $\frac{1}{2}=
  \frac{1}{p} + \frac{1}{s}$, which allows us to conclude that
  $\bar{q}_{\gamma} \nabla \bar{u}_\gamma \to \bar{q} \nabla \bar{u}$
  in $L^2 \left( \Omega, \mathbb{R}^2 \right)$,
  proving~\eqref{eq:extended_convergence_regularization:4}.

  Utilizing the strong $L^2(\Omega)$ convergence shown
  in~\eqref{eq:extended_convergence_regularization:4} we get
  \begin{equation*}
    \begin{aligned}
      \bar{\lambda}_{\gamma} =  \nabla \cdot \left( \bar{q}_\gamma \nabla \bar{u}_\gamma \right) + f  \to \nabla \cdot \left( \bar{q} \nabla \bar{u} \right) + f = \bar{\lambda} \text{ in } H^{-1} \left( \Omega \right),
    \end{aligned}
  \end{equation*}
  proving~\eqref{eq:extended_convergence_regularization:7}.

\end{proof}

Based on these limits it is easy to see that the limit of regularized
solution is admissible in the original problem.
\begin{theorem}
  Any limit point $(\bar{q}, \bar{u}, \bar{\lambda}) \in Q^{\rm{ad}} \times H^1_0 (\Omega) \times L^2 (\Omega)$ as obtained in
  Corollary~\ref{crl:extended_convergence_regularization}, fulfills
  \begin{subequations}
    \label{eq:optimality_conditions_obstacle_problem:All}
    \begin{align}
      - \nabla \cdot \left( \bar{q} \nabla \bar{u} \right) &= f +
      \bar{\lambda} & & \text{in } H^{-1}(\Omega), \label{eq:optimality_conditions_obstacle_problem:1}\\
      \bar{u} &\geq \psi &&\text{q.e. in } \Omega, \label{eq:optimality_conditions_obstacle_problem:2}\\
      \bar{\lambda} &\geq 0 &&\text{in } L^2(\Omega), \label{eq:optimality_conditions_obstacle_problem:3}\\
      \left( \bar{\lambda}, \bar{u} - \psi \right) &= 0, \label{eq:optimality_conditions_obstacle_problem:4}
    \end{align}
    with $\psi \in \mathbb{R}$ and $f \in L^\infty(\Omega)$.
  \end{subequations}
\end{theorem}

\begin{proof}
  The condition~\eqref{eq:optimality_conditions_obstacle_problem:1} is a result
  of Corollary~\ref{crl:extended_convergence_regularization}.

  The system of
  conditions~\eqref{eq:optimality_conditions_obstacle_problem:2}--\eqref{eq:optimality_conditions_obstacle_problem:4} is fulfilled
  because the relevant limits from
  Theorem~\ref{thm:convergence_regularization} and
  Corollary~\ref{crl:extended_convergence_regularization} allow us to
  proceed analogously to Lemma~\ref{lem:convergence_state_setcontrol}
  without the need for a given control.
\end{proof}

Further we can formulate the limit of the sequence of control equations of the regularized problems.
\begin{corollary}\label{cor:control_equation}
  If $\gamma \to \infty$, then there is a subsequence of solutions
  $(\bar{q}_\gamma , \bar{u}_\gamma) \in Q^{\rm{ad}} \times H^1_0 (\Omega)$ to
  problem~\eqref{eq:reg_obstacle_problem}, with corresponding adjoint
  $\bar{p}_{\gamma} \in H_0^1(\Omega)$ as defined in
  Proposition~\ref{prop:opt_conditions_regularized}, such that
  \begin{equation}\label{eq:extended_convergence_regularization_control}
    \begin{aligned}
      \nabla \bar{u}_{\gamma} \otimes \nabla \bar{p}_{\gamma} &\rightharpoonup \nabla \bar{u} \otimes \nabla \bar{p} &&\text{ in } L^s \left( \Omega, \mathbb{R}^{2 \times 2}_{\text{sym}} \right) \text{ for an } 1 < s < \infty,
    \end{aligned}
  \end{equation}
  for some $(\bar{u}, \bar{p}) \in
  H^1_0 (\Omega) \times H^1_0 (\Omega)$ and any such limit point fulfills
  \begin{equation}\label{eq:optimality_conditions_obstacle_problem_control}
    \begin{aligned}
      \left( \alpha \bar{q} - \nabla \bar{u} \otimes \nabla \bar{p} \right) (q - \bar{q}) &\geq 0 &&\forall q \in Q^{\rm{ad}}.
    \end{aligned}
  \end{equation}
  where $\bar{q} \in Q^{\rm{ad}}$ is the limit of the sequences of controls $\bar{q}_\gamma$ as specified in Theorem~\ref{thm:convergence_regularization}.
\end{corollary}

\begin{proof}
  For~\eqref{eq:extended_convergence_regularization_control} we note that, by the $W^{1,p} (\Omega)$
  convergence of $\bar{u}_\gamma$  from~\eqref{eq:convergence_regularization:2}  for some $p>2$ and the weak $W^{1,2} (\Omega)$
  convergence of $\bar{p}_\gamma$
  from~\eqref{eq:convergence_regularization:3}, for $s > 1$ satisfying
  $\frac{1}{s} = \frac{1}{2} + \frac{1}{p}$ it is
  \begin{equation*}
    \begin{aligned}
      (\nabla \bar{u}_\gamma \otimes \nabla \bar{p}_\gamma, w) \to (\nabla \bar{u} \otimes \nabla \bar{p}, w)
    \end{aligned}
  \end{equation*}
  for all $w \in L^{s'}(\Omega , \mathbb{R}_{\text{sym}}^{2 \times 2} )$.
  So we can conclude that
  $\nabla \bar{u}_{\gamma} \otimes \nabla \bar{p}_{\gamma} \rightharpoonup \nabla \bar{u} \otimes \nabla \bar{p}$
  in $L^s \left( \Omega, \mathbb{R}^{2 \times 2}_{\text{sym}} \right)$ for an $s \in (1, \infty)$ with limits
  $\bar{u}, \bar{p} \in H^1_0(\Omega)$ as specified in
  Corollary~\ref{crl:extended_convergence_regularization}.

  The
  equation~\eqref{eq:optimality_conditions_obstacle_problem_control}
  is then a result of the
  convergences~\eqref{eq:extended_convergence_regularization_control}
  and~\eqref{eq:convergence_regularization:1}.
\end{proof}

Finally we consider the limits with regards to the adjoint equation.
\begin{corollary}
  If $\gamma \to \infty$, then there is a subsequence of solutions
  $(\bar{q}_\gamma , \bar{u}_\gamma) \in Q^{\rm{ad}} \times H^1_0 (\Omega)$ to
  problem~\eqref{eq:reg_obstacle_problem}, with corresponding adjoint
  $\bar{p}_{\gamma} \in H_0^1(\Omega)$ as defined in
  Proposition~\ref{prop:opt_conditions_regularized} and multipliers
  $\bar{\lambda}_\gamma, \bar{\mu}_\gamma \in L^2 (\Omega)$ as defined in
  Corollary~\ref{crl:opt_conditions_regularized_extended}, such that
  \begin{subequations}
    \label{eq:extended_convergence_regularization_adjoint:All}
    \begin{align}
      \bar{q}_{\gamma} \nabla \bar{p}_{\gamma} &\rightharpoonup \bar{q} \nabla \bar{p} &&\text{ in } L^s \left( \Omega \right) \text{ for all } 1 \leq s < 2, \label{eq:extended_convergence_regularization_adjoint:1}\\
      \bar{\mu}_{\gamma} &\rightharpoonup \bar{\mu}  &&\text{ in } L^s \left( \Omega \right) \text{ for all } 1 \leq s < 2 \label{eq:extended_convergence_regularization_adjoint:2}
    \end{align}
  \end{subequations}
  for some $(\bar{q}, \bar{p}, \bar{\mu}) \in
  Q^{\rm{ad}} \times H^1_0 (\Omega) \times H^1_0 (\Omega) \times L^2 (\Omega)$ and any such limit point fulfills
  \begin{subequations}
    \label{eq:optimality_conditions_obstacle_problem_adjoint:All}
    \begin{align}
      - \nabla \cdot \left( \bar{q} \nabla \bar{p} \right) &= \bar{u} - u_d - \bar{\mu} & & \text{in } H^{-1}(\Omega), \label{eq:optimality_conditions_obstacle_problem_adjoint:1}\\
      \left( \bar{p} , \bar{\lambda} \right) &= 0, \label{eq:optimality_conditions_obstacle_problem_adjoint:2}\\
      \left( \bar{\mu}, \bar{u}-\psi \right) &= 0, \label{eq:optimality_conditions_obstacle_problem_adjoint:3}
    \end{align}
  \end{subequations}
  with limits $\bar{u} \in H^1_0 (\Omega)$ as specified in
  Theorem~\ref{thm:convergence_regularization} and $\bar{\lambda} \in L^2 (\Omega)$ as specified
  in Corollary~\ref{crl:extended_convergence_regularization} respectively.
\end{corollary}

\begin{proof}
  We note that
  \begin{equation*}
    \begin{aligned}
      - \nabla \cdot \left( \bar{q}_\gamma \nabla \bar{p}_\gamma \right) = -\partial_u r ( \gamma , \bar{u}_\gamma ) \bar{p}_\gamma + \bar{u} - u_d \in L^s(\Omega) \text{ for all } 1 \leq s < 2
    \end{aligned}
  \end{equation*}
  since $\partial_u r(\gamma, \bar{u}), u_d \in L^2(\Omega)$ and $\bar{u}_\gamma , \bar{p}_\gamma \in H^1_0 (\Omega)$.

  We begin considering~\eqref{eq:extended_convergence_regularization_adjoint:1}
  by utilizing~\eqref{eq:convergence_regularization:1}
  and~\eqref{eq:convergence_regularization:3}. Because $\bar{q}_\gamma$
  is convergent in $L^p ( \Omega , \mathbb{R}^2 )$ for all
  $p \in [2, \infty)$ and $\bar{p}_\gamma$ is weakly convergent
  in $W^{1,2} (\Omega)$, we can conclude that the product of these
  weak and strong convergences gives us the weak convergence
  $ \bar{q}_{\gamma} \nabla \bar{p}_{\gamma} \rightharpoonup \bar{q} \nabla \bar{p} $
  for all $1 \leq s < 2$ by, e.g.,~\cite[Lemma 1.2.4]{Allaire:2002}
  since we can always find a $p>2$ large enough so that
  $\bar{q}_\gamma$ converges in $L^p(\Omega)$ and
  $$\frac{1}{s} = \frac{1}{2} + \frac{1}{p}.$$
  Now we consider the convergence of $\bar{\mu}_\gamma$.

  To do so, we use the weak
  convergence of $\bar{q}_\gamma \nabla \bar{p}_\gamma$
  in $L^s \left( \Omega \right)$ for an $1 \leq s < 2$ and the strong
  convergence of $\bar{u}_\gamma$ in $L^2 ( \Omega )$ to get the
  convergence of the sequence~\eqref{eq:extended_convergence_regularization_adjoint:2}, i.e.,
  \begin{equation*}
    \begin{aligned}
      \bar{\mu}_{\gamma} =  \bar{u}_{\gamma} - u^d + \nabla \cdot \left(  \bar{q}_{\gamma} \nabla \bar{p}_{\gamma} \right) \rightharpoonup \bar{u} - u^d + \nabla \cdot \left( \bar{q} \nabla \bar{p} \right) =: \bar{\mu} \text{ in } L^s \left( \Omega \right)
    \end{aligned}
  \end{equation*}
  for all $1 \leq s < 2$ by the convergences given above.
  This also
  shows~\eqref{eq:optimality_conditions_obstacle_problem_adjoint:1},
  noting the additional regularity $\bar{\mu} \in H^{-1}(\Omega)$
  coming from~\eqref{eq:convergence_regularization:3} and the
  definition of $\bar{\mu}$.

  Finally we consider the set of conditions~\eqref{eq:optimality_conditions_obstacle_problem_adjoint:2} and \eqref{eq:optimality_conditions_obstacle_problem_adjoint:3}.

  We start by showing that $\left( \bar{\lambda} , \bar{p} \right) = 0$.
  The strong convergence of $\bar{\lambda}_{\gamma}$ in $H^{-1}\left(
    \Omega \right)$ from Corollary~\ref{crl:extended_convergence_regularization}
  and weak convergence of $\bar{p}_\gamma$ in $H^1_0 \left( \Omega \right)$ yield the convergence
  of the scalar product, see, e.g.,~\cite[Lemma
  1.2.4]{Allaire:2002}. Hence $(p_\gamma,\lambda_\gamma) \rightarrow (\lambda,p)$.
  To compute the value of the limit, we first note that we can utilize the asymptotic behavior of the state $u_\gamma$ with
  \begin{equation}\label{stateEstimate}
    \begin{aligned}
      \| \max (0, \psi - \bar{u}_\gamma) \|_{L^\infty(\Omega)} \lesssim \frac{1}{\gamma} (\| f \|_{L^\infty(\Omega)}+1)
    \end{aligned}
  \end{equation}
  by using Lemma~\ref{lem:growth_cond}.
  This allows us to estimate
  \begin{equation*}
    \begin{aligned}
      |(\lambda_\gamma,p_\gamma)| &= \Big| \bigl( \gamma \maxgamma (\psi - \bar{u}_\gamma), \bar{p}_\gamma \bigr) \Big|\\
      &\lesssim  \bigl( \gamma (\psi - \bar{u}_\gamma) {\maxgamma}' (\psi - \bar{u}_\gamma), |\bar{p}_\gamma| \bigr) \\
      &= \bigl( \max(0, \psi - \bar{u}_\gamma), \gamma {\maxgamma}' (\psi - \bar{u}_\gamma) |\bar{p}_\gamma| \bigr) \\
      &\leq \Big\| \max(0, \psi - \bar{u}_\gamma) \Big\|_{L^{\infty} ( \Omega )} \Big\| \gamma {\maxgamma}' (\psi - \bar{u}_\gamma) \bar{p}_\gamma \Big\|_{L^1(\Omega)}\\
      &\lesssim \frac{1}{\gamma} (\| f \|_{L^{\infty}(\Omega)}+1) \left\| u_\gamma - u_d \right\|_{L^1(\Omega)} \to 0.
    \end{aligned}
  \end{equation*}
  In the first inequality, we utilized $0\le \maxgamma (x) \leq c {\maxgamma}' (x) x$ for
  some $c \ge 1$. Then, we utilized
  $(\psi - \bar{u}_\gamma) = \max(0, \psi - \bar{u}_\gamma)$
  when ${\maxgamma}' (\psi - \bar{u}_\gamma) \not = 0$ and
  for the last inequality we used the bound given in~\eqref{stateEstimate} and the bound on $\mu_\gamma= \gamma {\maxgamma}' (\psi - \bar{u}_\gamma) \bar{p}_\gamma$ from Lemma~\ref{lem:mu_bound} in $L^1 (\Omega)$. Finally, since $u_\gamma- u_d$ is bounded, we get convergence for $\gamma \to \infty$.

  For~\eqref{eq:optimality_conditions_obstacle_problem_adjoint:3}, we use the definition of $\bar{\mu}_{\gamma}$
  in~\eqref{eq:extended_opt_conditions_regularized:4} tested with $( \psi - \bar{u}_{\gamma})$,
  which implies
  \begin{equation*}
    \begin{aligned}
      \left| \left( \bar{\mu}_{\gamma} , \psi - \bar{u}_{\gamma} \right) \right|
      &= \left| \left( \gamma {\maxgamma}' (\psi - \bar{u}_\gamma) \bar{p}_{\gamma},  \psi - \bar{u}_{\gamma} \right) \right| \\
      &\le \left( \gamma {\maxgamma}' (\psi - \bar{u}_\gamma)(\psi-\bar{u}_\gamma), |\bar{p}_{\gamma}| \right) \\
      &\lesssim  \left( \gamma \maxgamma (\psi - \bar{u}_\gamma),
        |\bar{p}_{\gamma}|\right) \\
      &=  ( \bar{\lambda}_\gamma, |\bar{p}_\gamma|) \rightarrow 0,
    \end{aligned}
  \end{equation*}
  noting that the proof
  of~\eqref{eq:optimality_conditions_obstacle_problem_adjoint:2}
  also implies $( \bar{\lambda}_\gamma, |\bar{p}_\gamma|) \rightarrow 0$.
\end{proof}

\bibliographystyle{plain}
\bibliography{bibfile}{}

\begin{thebibliography}{10}

\bibitem{Adams:2003}
R.~A. Adams and J.~J.~F. Fournier.
\newblock {\em Sobolev Spaces}, volume 140 of {\em Pure and Applied Mathematics
  (Amsterdam)}.
\newblock Elsevier/Academic Press, Amsterdam, second edition, 2003.

\bibitem{Allaire:2002}
G.~Allaire.
\newblock {\em Shape Optimization by the Homogenization Method}, volume 146 of
  {\em Applied Mathematical Sciences}.
\newblock Springer New York, New York, NY, 2002.

\bibitem{AsmannRoesch:2013}
U.~Aßmann and A.~Rösch.
\newblock Identification of an unknown parameter function in the main part of
  an elliptic partial differential equation.
\newblock {\em Zeitschrift für Analysis und ihre Anwendung}, 32(2), 2013.

\bibitem{AsmannRoesch:2015}
U.~Aßmann and A.~Rösch.
\newblock Regularization in sobolev spaces with fractional order.
\newblock {\em Numerical Functional Analysis and Optimization}, 36(3), 2015.

\bibitem{Barbu:1984}
V.~Barbu.
\newblock {\em Optimal Control of Variational Inequalities}, volume 100 of {\em
  Research Notes in Mathematics}.
\newblock Pitman, Boston, 1984.

\bibitem{Clarke:1976}
F.~H. Clarke.
\newblock A new approach to lagrange multipliers.
\newblock {\em Mathematics of Operations Research}, 1(2):165--174, 1976.

\bibitem{DeckelnickHinze:2011}
K.~Deckelnick and M.~Hinze.
\newblock Identification of matrix parameters in elliptic {PDEs}.
\newblock {\em Control Cybernet.}, 40(4):957--969, 2011.

\bibitem{DeckelnickHinze:2012}
K.~Deckelnick and M.~Hinze.
\newblock Convergence and error analysis of a numerical method for the
  identification of matrix parameters in elliptic {PDEs}.
\newblock {\em Inverse Problems}, 28(11), 2012.

\bibitem{Groger:1989}
K.~Gröger.
\newblock A {Wl},p-{Estimate} for {Solutions} to {Mixed} {Boundary} {Value}
  {Problems} for {Second} {Order} {Elliptic} {Differential} {Equations}.
\newblock {\em Mathematische Annalen}, 283(4), 1989.

\bibitem{HaslingerKocvaraLeugeringStingl:2010}
J.~Haslinger, M.~Kočvara, G.~Leugering, and M.~Stingl.
\newblock Multidisciplinary free material optimization.
\newblock {\em SIAM Journal on Applied Mathematics}, 70(7):2709--2728, 2010.

\bibitem{P19Report2023}
A.~Hehl, D.~Khimin, I.~Neitzel, N.~Simon, T.~Wick, and W.~Wollner.
\newblock Coefficient control of variational inequalities.
\newblock Preprint 2307.00869, arXiv, 2023.

\bibitem{HerzogMeyerWachsmuth:2011}
R.~Herzog, C.~Meyer, and G.~Wachsmuth.
\newblock Integrability of displacement and stresses in linear and nonlinear
  elasticity with mixed boundary conditions.
\newblock {\em Journal of Mathematical Analysis and Applications}, 382(2),
  2011.

\bibitem{Hintermueller:2001}
M.~Hinterm{\"u}ller.
\newblock Inverse coefficient problems for variational inequalities: Optimality
  conditions and numerical realization.
\newblock {\em ESAIM Mathematical Modelling and Numerical Analysis},
  35(1):129--152, 2001.

\bibitem{HornJohnson:1990}
R.A. Horn and C.R. Johnson.
\newblock {\em Matrix Analysis}.
\newblock Cambridge University Press, 1990.

\bibitem{ItoKunisch:2000}
K.~Ito and K.~Kunisch.
\newblock Optimal control of elliptic variational inequalities.
\newblock {\em Applied Mathematics and Optimization}, 41:343--364, 2000.

\bibitem{KinderlehrerStampacchia:2000}
D.~Kinderlehrer and G.~Stampacchia.
\newblock {\em An introduction to variational inequalities and their
  applications}, volume~31 of {\em Classics in Applied Mathematics}.
\newblock Society for Industrial and Applied Mathematics, Philadelphia, 2000.

\bibitem{KunischWachsmuth:2012a}
K.~Kunisch and D.~Wachsmuth.
\newblock Path-following for optimal control of stationary variational
  inequalities.
\newblock {\em Computational Optimization and Applications}, 51(3):1345--1373,
  2012.

\bibitem{MeyerRademacherWollner:2015}
C.~Meyer, A.~Rademacher, and W.~Wollner.
\newblock Adaptive optimal control of the obstacle problem.
\newblock {\em SIAM Journal on Scientific Computing}, 37(2):A918--A945, 2015.

\bibitem{Mignot:1976}
F.~Mignot.
\newblock Contrôle dans les inequations variationelles elliptiques.
\newblock {\em Journal of Functional Analysis}, 22(2):130--185, 1976.

\bibitem{MuratTartar:2018b}
F.~Murat and L.~Tartar.
\newblock H-convergence.
\newblock In A.~V. Cherkaev and R.~Kohn, editors, {\em Topics in the
  mathematical modelling of composite materials}, Modern Birkhäuser Classics,
  chapter~3. Birkhäuser Cham, 2018.

\bibitem{MuratTartar:2018a}
F.~Murat and L.~Tartar.
\newblock On the control of coefficients in partial differential equations.
\newblock In A.~V. Cherkaev and R.~Kohn, editors, {\em Topics in the
  mathematical modelling of composite materials}, Modern Birkhäuser Classics,
  chapter~1. Birkhäuser Cham, 2018.

\bibitem{NeitzelWickWollner:2017}
I.~Neitzel, T.~Wick, and W.~Wollner.
\newblock An optimal control problem governed by a regularized phase-field
  fracture propagation model.
\newblock {\em SIAM J. Control Optim.}, 55(4):2271--2288, 2017.

\bibitem{NeitzelWickWollner:2019}
I.~Neitzel, T.~Wick, and W.~Wollner.
\newblock An optimal control problem governed by a regularized phase-field
  fracture propagation model. {P}art {II}: The regularization limit.
\newblock {\em SIAM Journal on Control and Optimization}, 57(3):1672--1690,
  2019.

\bibitem{RaulsUlbrich:2019}
A.-T. Rauls and S.~Ulbrich.
\newblock Computation of a {B}ouligand generalized derivative for the solution
  operator of the obstacle problem.
\newblock {\em SIAM J. Control Optim.}, 57(5):3223--3248, 2019.

\bibitem{RaulsUlbrich:2021}
A.-T. Rauls and S.~Ulbrich.
\newblock On the characterization of generalized derivatives for the solution
  operator of the bilateral obstacle problem.
\newblock {\em SIAM J. Control Optim.}, 59(5):3683--3707, 2021.

\bibitem{RenardyRogers:2004}
M.~Renardy and R.~C. Rogers.
\newblock {\em An introduction to partial differential equations}.
\newblock Texts in applied mathematics. Springer New York, New York, 2004.

\bibitem{Rodrigues:1987}
F.~Rodrigues.
\newblock {\em Obstacle Problems in Mathematical Physics}.
\newblock North-Holland Mathematics Studies. Elsevier Science Publishers, 1987.

\bibitem{ScheelScholtes:2000}
H.~Scheel and S.~Scholtes.
\newblock Mathematical programs with complementarity constraints: Stationarity,
  optimality, and sensitivity.
\newblock {\em Mathematics of Operations Research}, 25(1):1--22, 2000.

\bibitem{SchielaWachsmuth:2011}
A.~Schiela and D.~Wachsmuth.
\newblock Convergence analysis of smoothing methods for optimal control of
  stationary variational inequalities.
\newblock {\em ESAIM Math. Model. Numer. Anal.}, 47(3):771--787, 2013.

\bibitem{SimonWollner:2023}
N.~Simon and W.~Wollner.
\newblock First order limiting optimality conditions in the coefficient control
  of an obstacle problem.
\newblock {\em PAMM}, 22(1), 2023.

\bibitem{Tartar:2018}
L.~Tartar.
\newblock Estimation of homoginized coefficients.
\newblock In A.~V. Cherkaev and R.~Kohn, editors, {\em Topics in the
  mathematical modelling of composite materials}, Modern Birkhäuser Classics,
  chapter~1. Birkhäuser Cham, 2018.

\bibitem{Troltzsch:2009}
F.~Tröltzsch.
\newblock {\em Optimale {Steuerung} partieller {Differentialgleichungen}:
  {Theorie}, {Verfahren} und {Anwendungen}}.
\newblock Studium. Vieweg + Teubner, 2., überarb. aufl edition, 2009.

\bibitem{Wachsmuth:2016}
G.~Wachsmuth.
\newblock Towards {M}-stationarity for optimal control of the obstacle problem
  with control constraints.
\newblock {\em SIAM Journal on Control and Optimization}, 54(2):964--986, 2016.

\end{thebibliography}

\end{document}